\documentclass{amsart}
\usepackage{amsmath,amsthm, amscd, amssymb, amsfonts}
\usepackage{enumerate,MnSymbol}
\usepackage{pb-diagram}

\newtheorem{teor}{Theorem}[section] \newtheorem{corol}[teor]{Corollary}
\newtheorem{prop}[teor]{Proposition} \newtheorem{lem}[teor]{Lemma}

\newtheorem{prop-def}[teor]{Proposition-Definition}

\theoremstyle{definition} \newtheorem{defin}[teor]{Definition}
 \newtheorem{ejem}[teor]{Example}

\theoremstyle{remark}
\newtheorem{obs}[teor]{Remark}\newtheorem{nota}[teor]{Remark}

\newcommand\M{\mathcal{M}}
\newcommand{\gap}{\textsf{GAP}}
\newcommand{\Sp}{\operatorname{Sp}}
\newcommand{\ASp}{\operatorname{ASp}}
\newcommand{\St}{\operatorname{St}}
\newcommand{\Ps}{\operatorname{Ps}}
\newcommand{\APs}{\operatorname{APs}}
\newcommand{\Alt}{\operatorname{Alt}}
\newcommand{\Bil}{\operatorname{Bil}}
\newcommand{\Quad}{\operatorname{Quad}}
\newcommand{\G}{\mathcal{G}}
\newcommand{\anupq}{\textsf{ANUPQ}}
  \newcommand\s{\sigma}\newcommand
\End{\operatorname{End}}\newcommand\id{\operatorname{id}}
 \newcommand\Tr{\operatorname{Tr}}
 \newcommand\co{\operatorname{co}}
 
 \newcommand\Alg{\operatorname{Alg}}
 
\newcommand\Hom{\operatorname{Hom}} \newcommand\Rep{\operatorname{Rep}}

\newcommand\Ind{\operatorname{Ind}}\newcommand\Aut{\operatorname{Aut}}
\newcommand\Gal{\operatorname{Gal}} 
 \newcommand\can{\operatorname{can}}

\begin{document}

\title{Isocategorical groups and their  Weil representations} \author{C\'esar
Galindo} \address{Departamento de Matem\'aticas,
Universidad de los Andes, Bogot\'a, Co\-lom\-bia}
\email{cn.galindo1116@uniandes.edu.co, cesarneyit@gmail.com}

\begin{abstract}
Two groups are called isocategorical over a field $k$ if their res\-pective categories of $k$-linear representations are monoidally equivalent. We classify isocategorical groups over arbitrary fields, extending the earlier classification of Etingof-Gelaki and Davydov  for algebraically closed fields. In order to construct  concrete examples of isocategorical groups a new variant of the Weil representation associated to isocategorical groups is defined. We construct examples of non-isomorphic isocategorical groups over any field of characteristic different from two  and rational Weil representations associated to symplectic spaces over finite fields of characteristic two. 
\end{abstract}

\maketitle
\setcounter{tocdepth}{1}
\tableofcontents
\section{Introduction}

\textbf{1.}\  Two groups are called isocategorical over a field $k$ if their respective categories of $k$-linear representations are monoidally equivalent. The goals of this paper are to classify isocategorical groups over arbitrary fields, extending a previous classification of Etingof-Gelaki \cite{isocategorical} and Davydov \cite{Davydov}  for algebraically closed fields and to introduce a new variant of the Weil representation associated to simple Galois algebras and isocategorical groups. These Weil representations include as particular cases the unitary Weil representations associated to symplectic and quadratic spaces over a finite fields of characteristic two.

\textbf{2.}\  Let $G$ be an algebraic  group over a field $k$. The category $\Rep_k(G)$ of finite dimensional $k$-linear representations is a symmetric tensor category over $k$. Symmetric tensor categories of this kind can be characterized using Tannaka duality theory, as those admitting a symmetric fiber functor, \cite{Saavedra}. Moreover, if $k$ is algebraically closed of characteristic $0$, the \textit{symmetric} tensor category $\Rep_k(G)$ determines the group $G$ up to isomorphism, see \cite[Theorem 3.2]{Deligne-Milne}. Despite this, there exist examples of non isomorphic isocategorical groups, for example,  the affine symplectic group $\ASp(V):= V\rtimes \Sp(V)$ and the pseudo-symplectic group $\APs(V)$ ($V$ a symplectic space over the field of two elements ) are non-isomorphic isocategorical groups over $\mathbb{C}$, \cite{Davydov, isocategorical}  (however, they are not isocategorical  over $\mathbb{R}$, see Proposition \ref{Prop:No isocategorical sobre R}). 

\textbf{3.}\  The Weil representation is a unitary projective representation of the symplectic group over a local field, \cite{pseudosymplectic}. When the local field is non-archimedean with residual field of characteristic two the Weil representation is a real projective representation of $\Ps(V)$ the pseudo-symplectic group, see \cite{Weil-fini-adv}. Isocategorical groups over arbitrary fields are closely related to Weil representations, for example, the affine orthogonal group and the pseudo-symplectic group of a quadratic space over a finite field of characteristic two are non-isomorphic isocategorical over  $\mathbb{Q}$, see Proposition \ref{Prop:Pseudo simplect y affine orthogonal son isocat}.

As was pointed in \cite{Weil-fini-adv} and \cite{isocategorical} the unitary Weil representation of $\APs(V)$ is an extension of the Weil representation of $\Ps(V)$. However, the Weil representation of $\APs(V)$ is not real. Is natural to ask if there exists a \textit{real} Weil representation associated to a symplectic space over a finite field of characteristic  two. We give a positive answer to this question and provide some examples of isocategorical groups over $\mathbb{Q}$ of a nature slightly different from the affine orthogonal and pseudo-symplectic group.

\textbf{4.}\  We would like to finish the introduction by pointing out an interesting relation between isocategorical groups and stringy orbifold theory.  The  Drinfel'd double $D(k[G])$  of a group algebra plays an important roll in stringy orbifold theory, for example (see \cite{KP} and \cite{MS}):
\begin{itemize}
\item The category of $G$-Frobenius algebras arising in global orbifold cohomology or $K$-theory  is the category of Frobenius algebras
in the modular category of finite dimensional $D(k[G])$-modules.
\item The Grothendieck ring of the modular category of finite dimensional $D(k[G])$-modules can be realized geometrically as $K_{orb}^{k}([*/G])$  the stringy $K$-theory of the orbifold  $[*/G]$.
\end{itemize}

The category of representation of the Drinfel'd double has a conceptual interpretation as the Drinfel'd center of the category Rep$_k(G)$ of $k$-linear representations of $G$. Thus, the Drinfel'd doubles of isocategorical groups are braided equivalents. Then as application of our main results we explicitly construct some family of example of pairs of non-isomorphic groups with the same category of $G$-Frobenius algebras and same  stringy $K$-theories for every field $k$ of characteristic different to two.

\textbf{5.}\  The paper is organized as follows: In Section \ref{preliminare} we discuss the classification of Galois algebras of finite groups  and their relation with  isocategorical groups. In Section \ref{Seccion:Isocategorical groups over arbitrary fields} we give a classification of isocategorical groups over arbitrary fields and we study in detail isocategorical groups over formally real fields. The main results of this section are Theorem \ref{main result 2} and Corollary \ref{Corol Clasification over Q nad Z[1/2]}. In Section  \ref{Section:Galois algebras and  Weil representations} we introduce the Weil representation associated to  a simple Galois algebra over a finite abelian group. We develope in detail the case of Weil representations associated to arbitrary symplectic and quadratic modules. In Section \ref{Section:Examples of non-isomorphic isocategorical groups and  Weil representations} we present some concrete examples of non-isomorphic isocategorical groups over $\mathbb{Q}$ and an application to the Weil representation associated to a quadratic space over a finite field of characteristic two.

\section{Preliminaries}\label{preliminare}

Our main reference for the general theory of Hopf algebra is \cite{Mont}. We make free use of Sweedler's notation for comultiplcations and comodule structures, omitting the summation symbols: $\Delta(c)=c_1\otimes c_2$,  $\rho(v)=v_0\otimes v_1$ for right  comodules and $\lambda(v)=v_{-1}\otimes v_0$ for left comodules.

\subsection{Galois objects of Hopf algebras and Galois algebras of finite groups}\label{prels}

In this section we review some definition and results on Hopf Galois extensions
that we will need later. We refer the reader to \cite{galois-survey-schauenburg} and \cite{survey-susan} for a detailed exposition on the subject.

\begin{defin}
Let $k$ be a commutative base ring and $H$ a Hopf algebra over $k$.  A \emph{right $H$-Galois object} is a right $H$-comodule algebra $(A,\rho)$ such that $A$ is a faithfully flat $k$-module, $A^{\co H}=k$ and the canonical map
$$\can: A\otimes A\to A\otimes H, \  \  x\otimes y\mapsto xy_{(0)}\otimes y_{(1)},$$ is bijective.
Left $H$-Galois objects are defined similarly.
\end{defin}
A morphism of $H$-Galois object is an $H$-colinear algebra map. It is known that a morphism of $H$-Galois objects is an isomorphism.

\begin{defin}
An $(L, H)$-Bigalois object is an $(L, H)$-bicomodule algebra $A$
which is simultaneously a left $L$-Galois object and right
$H$-Galois object.
\end{defin}

For any right $H$-Galois object $A$ there is an associated Hopf
algebra $L = L(A, H)$, called the \emph{left Galois} Hopf algebra,
such that $A$ is in a natural way an $(L, H)$-Bigalois object.

\begin{obs}\label{obs: obj cleft en dimension finita} An $H$-Galois object $A$ is called \emph{cleft} if there is an $H$-colinear convolution invertible map $H\to A$.  In case that $A$ is cleft, the Hopf algebra  $L(A, H)$ is obtained by a \emph{cocycle deformation}  \cite[Theorem 3.9]{BiGal-Schau}, in particular for  finite dimensional Hopf algebras over  a field every Galois object is cleft \cite[Proposition 2]{cleft-finito}, so $\dim_k(L(A,H))=\dim_k (H)$. \end{obs}

If $G$ is  a finite group, we will denote by  $\mathcal{O}_k(G)$, the Hopf algebra of regular function of the constant group scheme over  $k$, that is,  $\mathcal{O}_k(G)$ is the free $k$-module with a basis  $\{\delta_g\}_{g \in G}$,   multiplication $m(\delta_g,\delta_h)=\delta_{g,h}\delta_g$ and comultiplication $\Delta(\delta_g)=\sum_{x,y \in G: xy=g}\delta_x\otimes \delta_y$, for all $g,h \in G$.

The category of  $\mathcal{O}_k(G)$-comodule algebras is the same as the category of $k$-algebras endowed with an action of $G$ by $k$-algebra automorphisms. If $A$ is a $G$-algebra, $A^{\co \mathcal{O}_k(G)}=A^G$, the subalgebra of $G$-invariants.

For simplicity, a $\mathcal{O}_k(G)$-Galois object will be called just a $G$-Galois algebra over $k$ or just a $G$-Galois algebra if the base ring is clear. Analogously a $(G_1,G_2)$-Bigalois algebra is just a $(\mathcal{O}_k(G_2)-\mathcal{O}_k(G_1))$-Bigalois object.

\subsection{Isocategorical groups  and Bigalois Algebras}

We will use freely the basic language of monoidal categories theory, for more reference see \cite{Bak-Kir} and \cite{Mac-Lane}.

If $H$ is a Hopf algebra we denote by $\M^H$ the $k$-linear monoidal category of all rigth $H$-comodules.

Given a finite group $G$,  we will denote by Rep$_k(G)$, the monoidal category of all (left) $k[G]$-modules. Note that Rep$_k(G)=\M^{\mathcal{O}_k(G)}$ the tensor category of $k$-linear representations of $G$.
\begin{defin}[\cite{isocategorical}]
Let $G_1$ and $G_2$ be two finite groups. We say that $G_1$ and $G_2$ are isocategorical over a commutative ring $k$ if the monoidal categories $\Rep_k(G_1)$ and $\Rep_k(G_2)$ are equivalents as monoidal $k$-linear categories.
\end{defin}
If $A$ is a $G$-algebra over $k$, the group of $G$-equivariant algebra automorphisms we will denote by $$\Aut_G(A)=\{f\in \Aut_{\text{Alg}}(A)| f(g \cdot a)=g\cdot f(a), \  \text{for all \ } g \in G\}.$$

\begin{lem}\label{lema Aut_G}
If $k$ is a field and $A$ is a $G$-Galois algebra, then $A$ is a $(G-\Aut_{G}(A))$-Bigalois algebra  if and only if $|\Aut_{G}(A)|=|G|$.
\end{lem}
\begin{proof}

If $A$ is a $G$-Galois algebra, there is a unique Hopf algebra (up to isomorphisms), such that $A$ is a $L$-$\mathcal{O}_k(G)$-Bigalois object, \cite[Theorem 3.5]{BiGal-Schau}. Remark  \ref{obs: obj cleft en dimension finita} implies that $|G|=\dim_k(L)$.

By \cite[Corollary 3.1.4]{galois-survey-schauenburg}, $\Alg(L, k) \simeq \Aut_{G}(A)$, so $L\cong \mathcal{O}_k(\Aut_{G}(A))$ if and only if $|G|=|\Aut_{G}(A)|$.
\end{proof}

\begin{prop}\label{Prop equivalencias con isocat  existencia bigalois}
Let $G_1$ and $G_2$ be two finite groups and $k$ be a commutative ring. The following are equivalents:
\begin{enumerate}
  \item[(a)] $G_1$ and $G_2$ are isocategoric.
  \item[(b)] There is a $(G_1,G_2)$-Bigalois algebra $A$ over $k$.
  \item[(c)] If $k$ is a field, $(a)$ and $(b)$ are equivalent to: there is a $G_1$-Galois algebra $A$ such that $\Aut_{G1}(A)\cong G_2$ and $|G_1|=|G_2|$.
\end{enumerate}
\end{prop}
\begin{proof}
Equivalence $(a)$ and $(b)$ is \cite[Corolario 5.7]{BiGal-Schau} and the equivalence between $(b)$ and $(c)$ follows by Lemma \ref{lema Aut_G}.

\end{proof}

\subsection{Group cohomology}

In order to fix notations,  we will recall the usual description of group cohomology associated to the normalized Bar resolution of $\mathbb{Z}$, see \cite{Eilenberg-MacLane} for more details.  Let $N$ be a group and let $A$ be a $\mathbb{Z}[N]$-module written in multiplicative notation. Define $C^0(N,A)=A$, and for $n \geq 1$ $$C^n(N,A)=\{f:\underbrace{N\times\cdots \times N}_{n-times}\to A| f(x_1\ldots,x_n)=1, \text{ if } x_i=1_N \text{ for some }i \}.$$
Considere the following cochain complex
\begin{equation*}\label{complex}
0 \longrightarrow C^0 (N, A) \stackrel{\delta_0}{\longrightarrow }
C^1 (N, A) \stackrel{\delta_1}{\longrightarrow }C^2 (N, A) \cdots C^{n} (N, A)
\stackrel{\delta_n}{\longrightarrow } C^{n+1} (N, A) \cdots
\end{equation*} where
\begin{align*}
    \delta_n(f)(x_1,x_2,\ldots,x_{n+1})&=x_1\cdot f(x_2,\ldots,x_{n+1})\\
    &\times\prod_{i=1}^nf(x_1,\ldots,x_{i-1},x_ix_{i+1},x_{i+2},\ldots,x_{n+1})^{(-1)^{i}}\\
    &\times f(x_1,\ldots,x_{n})^{(-1)^{n+1}}.
\end{align*}
Then,  $Z^n(N,A):=\ker(\delta_n)$, $B^n(N,A):= \text{Im}(\delta_{n-1})$ and
$$H^n(N,A):=Z^n(N,A)/B^n(N,A) \  \  (n\geq 1),$$ is the group cohomology of $N$ with coefficients in $A$.

\subsubsection{Group cohomology associated to a group exact sequence}
Let  $$1\to N\to S\to Q\to 1$$ be a group exact sequence. Let $A$ be a $Q$-module, consider $A$ as a $N$-module with the trivial action and $(C^{n} (N,A),\delta_n)$ as a  cochain  complex of $S$-modules, where $C^n(N,A)$ has $S$-module structure given by $$(^g f)(x_1,\ldots, x_n)= \ ^g f(\ g^{-1} x_1g,\ldots,\  g^{-1} x_ng),$$ for all $x_1,\ldots, x_n\in N, g \in  S, f\in C^n(N,A)$.  Since the maps $ \delta_n:C^{n} (N, A) \to C^{n+1} (N, A)$ are $S$-equivariant we have the double cochain complex $C^{p,q}_{S}(N,A):=C^q(S,C^{p+1}(N,A))$, $p,q\geq 0$, then we define the complex $$C_S^n(N,A):=\text{Tot}^{n-1}(C^{*,*}_{S}(N,A)), \  \  \text{for } n>1,$$ and  the  cohomology groups $$H^n_S(N,A):= H^n(C^{*}_{S}(N,A)), \  \  \  n\geq 0.$$

For future reference it will be useful to describe the equations that define a 2-cocycle and the coboundary of a 1-cochains:

\begin{itemize}
\item The 1-cochains are  $C_S^1(N,A)=C^1(N,A)$ and the 2-cochains are $C_S^2(N,A)=C^2(N,A)\oplus C^1(S,C^1(N,A))$, so a 2-cochain is a pair of normalized functions  $\sigma: N\times N\to A, \gamma: S\times N\to A$.
\item   A 2-cocycle is a 2-cochain $(\sigma, \gamma)$ such that
\begin{align}
\sigma(x,y)\sigma (xy,z)&=\sigma(y,z)\sigma(x,yz),\label{C1}\\
^g\sigma(x,y)\gamma(g,xy)&=\sigma (^gx,^gy)\gamma(g,x)\gamma(g,y),\label{C2}\\
\gamma(gh,x)&=\ ^g \gamma(h,x)\gamma(g,^hx),\label{C3}
\end{align}
for all $x,y\in N, g,h\in S$.

\item The coboundary  of a 1-cochain $\gamma:N\to A$ is given by
\begin{align*}
\partial (\gamma)(x,y)=\frac{\gamma(y)\gamma(x)}{\gamma(xy)},\\
\partial(\gamma)(g,x)= \frac{^g \gamma(x)}{\gamma(^g x)},
\end{align*}for all $g\in S, \ x,y\in N$.
\end{itemize}
\begin{prop}\label{Prop:chomologia normalizada}
Let  $ Z^2_S(N,A)_{n}$ be the subgroup of all 2-cocycle  $(\sigma,\gamma)$  such that $$\gamma(x,y)=\sigma(x,y)\sigma(xy,x^{-1})\sigma(x,x^{-1})^{-1}$$ for all $x, y\in N$.   Then $B^2_S(N,A)\subset Z^2_S(N,A)_{n}$ and $$H^2_S(N,A)\cong Z^2_S(N,A)_{n}/B^2_S(N,A).$$
\end{prop}

\begin{proof}
A straightforward  calculation shows that $B_S(N,A)\subset Z^2_S(N,A)_n$.

Note that if  $Q=1$, then $H^n_{N}(N,A)=H^n(N,A)$, a particular  quasi-isomorphism  $Z^2(N,A)\to Z^2_N(N,A)$ is given by $\sigma\mapsto (\sigma, \gamma)$, where $$\gamma(x,y)=\sigma(x,y)\sigma(xy,x^{-1})\sigma(x,x^{-1})^{-1},$$ for all $x, y\in N$.
Let $(\sigma, \gamma) \in Z^2_S(N,A)$
be an arbitrary 2-cocycle. Since $(\sigma,\gamma|_{N\times N})\in Z^2_N(N,A)$, there is $\gamma:N\to A$ such that $\partial(\gamma) (\sigma,\gamma)=(\sigma,\gamma_\sigma)\in Z^2_N(N,A)$. Then $\partial(\gamma) (\sigma,\gamma)\in Z^2_S(N,A)$ is a 2-cocycle in $Z^2_S(N,A)_n$ cohomologous to $(\sigma,\gamma)$.
\end{proof}

\subsubsection{The second cohomology group for abelian groups}

Let $V$ be a finite abelian group and let $k$ be a field such that $k^*$ is a divisible group.  A bicharacter $\omega: V\times V\to k^*$ is called a skew-symmetric form  if $\omega(x,x)=1$ for all $x\in V$. Let us denote by $\bigwedge^2 \widehat{V}$ the abelian group of all skew-symmetric forms over $k^*$.

\begin{prop}\label{Prop:exis de 2-cociclo para simplectco}
The group morphism 
\begin{align*}
\operatorname{Alt}:Z^2(V,k^*) \to \bigwedge^2 \widehat{V}\\
 \alpha \mapsto [(x,y)\mapsto \frac{\alpha(x,y)}{\alpha(y,x)}]
\end{align*}induces an isomorphism $H^2(V,k^*)\cong \bigwedge^2 \widehat{V}$.
\end{prop}
\begin{proof}
See \cite[Proposition 2.6]{Tambara-Func}.
\end{proof}
\subsection{Galois algebras over arbitrary fields}
In this section we review the main result of  \cite{Manuel}. In this section $k$ will denote a field.

Let $S$ be a subgroup of $G$ and $B$ be  an $S$-algebra. The induced algebra is defined as the algebra of functions \[ \Ind_S^G(B)=\{r:G\to B|r(sg)=s\cdot
r(g) \quad \forall s\in S,g\in G\}, \] with $G$-action
 $(g\cdot r)(x)=r(xg)$, for $r \in \Ind_S^G(B), g,x\in G$.

\begin{nota}\label{nota ind is a functor}
 Induction is a covariant functor, where each homomor\-phism
of $S$-algebras $f:A\to B$ is sent to the homomorphism of $G$-algebras
$\Ind_S^G(f):\Ind_S^G(A)\to \Ind_S^G(B)$, $\Ind_S^G(f)(r)=f\circ r$.
\end{nota}

 Let $S$ be a finite group, let $k$ be a field and let $\s\in Z^2(S, k^*)$ be a 2-cocycle.
For each $s\in S$, we will use the notation $u_s \in k_{\s}[S]$ to indicate the corresponding element in the
\emph{twisted group algebra} $k_{\s}[S]$. Thus $\{u_s\}_{s\in S}$ is a $k$-basis of
$k_{\s}[S]$, and in this basis $u_su_t = \s(s, t)u_{st}$ for all $s,t\in S$.

\begin{defin}
An element $s\in S$ is called $\s$-\emph{regular} if $\s(s, t) =
\s(t, s)$ for all $t \in C_S(s)$.  The 2-cocycle $\s$ is called
\emph{non-degenerate} if and only if $1\in S$ is the only
$\s$-regular element.
\end{defin}

\begin{obs}
If $S$ is an abelian group a 2-cocycle $\s\in Z^2(S, k^*)$ is non-degenerate if and only if the skew-symmetric form $\Alt(\s)$ is a non-degenerate bicharacter.
\end{obs}

\begin{defin}Let $G$ be a finite group and  $k$ be a field.
 A \emph{Galois datum associated to $G$} is a collection
$(S,K,N,\s,\gamma)$ such that
\begin{enumerate}
[$i)$] \item $S$ is a subgroup of $G$ and $N$ is a normal subgroup of $S$,
\item $K\supseteq k$ is a Galois extension with Galois group $S/N$,
\item char$(k)\nmid |N|$,
\item $(\sigma,\gamma)\in Z^2_{S}(N,K^*)_n$ is
2-cocycle such that $\sigma \in Z^2(N,K^*)$ is non-degenerate.
\end{enumerate}
\end{defin}

Let $(S, K,N,\s,\gamma)$ be a Galois datum associated to $G$. We will denote
by $A(K_\sigma[N],\gamma)$ the twisted group algebra $K_\sigma [N]$  with
$S$-action defined by \begin{equation*}
g\cdot (\alpha u_x)=(g\cdot \alpha)(g\cdot
u_x)=\overline{g}(\alpha)\gamma(g,x)u_{\-^gx}, \end{equation*} for $g\in S$, $x\in
N$, and $\alpha\in K$.

We will denote by $\Ind_S^G(A(K_\sigma[N],\gamma))$ the induced $G$-algebra from
the $S$-algebra $A(K_\sigma [N],\gamma)$.

\begin{obs}
If $(S, K,N,\s,\gamma)$ is a Galois datum where $K=k$, then $S=N$ and $\gamma$ is determined by $\sigma$. Then a Galois datum with $K=k$ is really just a pair $(S,\sigma)$ where $S$ is a subgroup of $G$ and $\sigma\in Z^2(S,k^*)$ is a non-degenerate 2-cocycle.
\end{obs}
Now we can reformulate the main results of \cite{Manuel}.

\begin{teor}\label{main result} Let $G$ be a finite group and let $k$ be a
field.

\begin{enumerate}[i)]
  \item Let $(S,K,N,\s,\gamma)$ be a Galois datum associated to $G$. Then
      the $G$-algebra $\Ind_S^G(A(K_\sigma[N],\gamma))$ is a $G$-Galois
      algebra over $k$.
  \item Let $A$ be a $G$-Galois algebra over $k$. Then $A\simeq \Ind_S^G(A(K_\sigma[N],\gamma))$ for a Galois datum $(S,K,N,\s,\gamma)$.
\end{enumerate} \end{teor}
\qed

\section{Isocategorical groups over arbitrary fields}\label{Seccion:Isocategorical groups over arbitrary fields}

Considering Proposition \ref{Prop equivalencias con isocat
existencia bigalois} and Lemma \ref{lema Aut_G}, in order to
construct all groups isocategorical to a fixed group
$G$ is enough to describe the Galois data of $G$ such that
$|\Aut_G(A)|=|G|$, where $A$ is the associated $G$-Galois algebra.

\subsection{Equivariant automorphisms of Galois objects}

 Let $G$ be a finite group, $S$ a subgroup of $G$ and
$(B,\cdot)$ an $S$-algebra. For each $g\in G$, we consider the
$g^{-1}Sg$-algebra $(B^{(g)},\cdot_g)$, where $B^{(g)}=B$ as algebras and
$g^{-1}Sg$-action given by $$h\cdot_g b=(g hg^{-1})\cdot b,$$ for all $h\in
g^{-1}Sg$ and $b\in B^{(g)}$.

Note that for all $g\in G$,   the map
\begin{align}
  \psi_g:\Ind_S^G(B)&\to \Ind_{g^{-1}Sg}^G(B^{(g)})\label{lema isomorfismo Ind_S^G(A) and Ind_g(-1)Sg^G(A^(g))}\\
f&\mapsto \psi_g(f)=[h\mapsto f(gh)].\notag \end{align} is a $G$-algebra isomorphism.

For a pair $(b,g)\in B\times G$, we define an element $\chi_{g}^b\in
\Ind_S^G(B)$ as the function
$$\chi_{g}^b(x)=\left\{
\begin{array}{ll}
     0, & \hbox{if $x\notin Sg$} \\
    s\cdot b, & \hbox{ if $x=sg$, where $s\in S$.}
\end{array}
\right.
$$
Note that $(b,g), (s',g')\in B\times G$ define the same element in
$\Ind_{S}^G{(B)}$ if and only if there is $s\in S$ such $g'=sg$ and
$b'=s\cdot b$.

The elements $\chi_g^1$ are central idempotents
that only depend of the coset $Sg$ and we will denote by
$\chi_{Sg}$. The action of $G$ on $\Ind_S^G(B)$ defines a transitive action of
$G$ on  $\{\chi_{x}\}_{x\in S/G}$ by $h\cdot \chi_{Sg}=
\chi_{Sgh^{-1}}$. If $B$ is a simple algebra the central primitive
idempotents of $\Ind_S^G(B)$ are exactly $\{\chi_{x}\}_{x\in S/G}$.
Therefore, in this case each $G$-algebra automorphism of $\Ind_S^G(B)$
determines a unique automorphism of the $G$-set $\{\chi_{x}\}_{x\in
S/G}$, and a group homomorphism
\begin{align}
    \pi: \Aut_G(\Ind_S^G(B))&\to N_G(S)/S \label{definicion de pi}\\
    F&\mapsto Sg \notag
\end{align}
where $Sg$ is the unique coset such that $F(\chi_S)=\chi_{Sg}$.

The following proposition can be seen as a generalization of \cite[Theorem 5.5]{Davydov}.
\begin{prop}\label{Prop: exct sequence Aut_G}
 Let $G$ be a finite group, $S$ a subgroup of $G$ and
$(B,\cdot)$ be a  $S$-Galois algebra. Then, the sequence
\begin{equation}\label{primera sucesion exacta}
1\to \Aut_S(B) \to \Aut_G(\Ind_S^G(B))\to N_G(S)/S,
\end{equation}is exact. The map $\Aut_S(B) \to \Aut_G(\Ind_S^G(B))$ is the
induction and $\Aut_G(\Ind_S^G(B))\to N_G(S)/S$ is the group
morphism \eqref{definicion de pi}. Moreover,
the sequence \begin{equation*}
1\to \Aut_S(B) \to \Aut_G(\Ind_S^G(B))\to N_G(S)/S\to 1,
\end{equation*}is exact if and only if $B\cong B^{(g)}$ for all $g\in N_G(S)$.
\end{prop}
\begin{proof}
First, we will see that induction is injective. Let
$T\in \Aut_S(B)$ such that $\Ind_S^G(T)=\id_{\Ind_S^G(B)}$. Then, for
all $r\in \Ind_S^G(B)$, $T\circ r=r$. In particular, for
$\chi_{e}^b$, we have $T\circ \chi_{S}^b(e) =\chi_{S}^b(e)$, so
$T(b)=b$ and this is true for all $b\in B$. Thus $T=\id_B$.

Now, let $F\in \Aut(\Ind_S^G(B))$ such that $\pi(F)=S$, that is, $F(\chi_{Sg})=\chi_{Sg}$ for all $g\in G$. Using the injective
map
\begin{align*}
    B&\to \Ind_S^G(B)\\
    b&\mapsto \chi_e^b,
\end{align*}we can and will identify $B$ with its image in $\Ind_S^G(B)$. Note that $f \in
B\subset \Ind_S^G(B)$ if and only if $\chi_{Sg}f=0$ for all $g\notin
S$. Thus,  $F(\chi_e^b)\in B$, because
$\chi_{Sg}F(\chi_e^b)= F(\chi_{Sg})F(\chi_e^b)=
F(\chi_{Sg}\chi_e^b)=0$ if $g\notin S$. Therefore, $F$ defines an
automorphism $F|_B:B\to B$, by $F(\chi_e^b)=\chi_e^{F|_B(b)}$. We
will see that $F=\Ind_S^G(F|B)$. Let
\begin{align*}
    \Ind_S^G(F|_B)(\chi_g^b)&= F|_B\circ \chi_g^b\\
                    &= F|_B(g^{-1}\cdot\chi_e^b)\\
                    &=g^{-1}F(\chi_e^b)\\
                    &=F(g^{-1}\chi_e^b)\\
                    &=F(\chi_g^b).
\end{align*}Since every $f\in \Ind_S^G(B)$ is a sum of elements of
the form $\chi_g^b$, it follows that $\Ind_S^G(F|_B)= F$.

Now we want to  show that $\pi$ is surjective if and only if $B\cong
B^{(g)}$ for all $g\in G$. Suppose that $\pi$ is surjective. Then, for any $g\in N_G(S)$ there
exists $F_g\in \Aut_G(\Ind_S^G(B))$ such that $\pi(\chi_S)=\chi_{Sg}$.
Using \eqref{lema isomorfismo Ind_S^G(A) and Ind_g(-1)Sg^G(A^(g))},
we have $G$-algebras isomorphisms
$$\Ind_S^G(B)\stackrel{F_g}{\longrightarrow}\Ind_{S}^G(B)
\stackrel{\psi_g}{\longrightarrow}\Ind_{S}^G(B^{(g)}),$$ and the
restriction $(\psi_g \circ F_g)|_B:B\to B^{(g)}$ defines an $S$-algebra
isomorphism. Conversely, if $\gamma_g:
B^{(g)}\to B$ is an $S$-algebra isomorphism, then
$$\Ind_S^G(B)\stackrel{\psi_g}{\longrightarrow}\Ind_{S}^G(B^{(g)})
\stackrel{\Ind_S^G(\gamma_g)}{\longrightarrow}\Ind_{S}^G(B),$$ is an
algebra isomorphism such that
$\pi(\Ind_S^G(\gamma_g)\circ\psi_g)=Sg$.
\end{proof}

\begin{corol}\label{corol: condiciones igual orden}
Let $G$ be a finite group and $S\subset G$ be a subgroup. Let $B$ be a simple
$S$-Galois algebra and $A= \Ind_S^G(B)$. Then $|\Aut_G(A)|=|G|$ if
and only if
\begin{enumerate}
  \item $|\Aut_S(B)|=|S|$,
  \item $S$ is a normal subgroup of $G$,
  \item for all $g\in G$, $B^{(g)}\cong B$ as $S$-algebras.
\end{enumerate}
\end{corol}
\begin{proof}
By Proposition \ref{Prop: exct sequence Aut_G} we have,

\begin{align*}
   |\Aut_G(A)| &= |\Aut_S(B)||\text{Im}(\pi)|\\
                     &\leq   |\Aut_S(B)||N_G(S)/S|\\
                     &\leq |S|\frac{|N_G(S)|}{|S|}\\
                     & = |N_G(S)|\leq |G|.
\end{align*}

Thus, if $|\Aut_G(A)|=|G|$, $S$ is a normal subgroup, $|$Aut$_S(B)|=|S|$ and Im$\pi=G/S$.
Therefore, by  Proposition \ref{Prop: exct sequence Aut_G}, $B^{(g)}\cong B$ for all $g\in G$.

Conversely, the third condition implies  that $|\Aut_G(A)|= |\Aut_S(B)||N_G(S)/S|$. Since $B$
is normal and  $|\Aut_S(B)|= |S|$, we have $|G|=|\Aut_G(A)|$.

\end{proof}

\begin{defin}\label{defin:torsor data}
Let $G$ be a finite group. A Galois data  $(S,K,N,\sigma,\gamma)$ will be called a \textit{ torsor data} if

\begin{enumerate}
\item $S$ is abelian 
\item $S=N\oplus \Gal(K|k)$,
\item $k$ has a primitive root of unity of order equal to the exponent of $N$,
\item $\sigma$ takes values in $k^*$,
\item $\gamma: S\times N\to k^*$ is a pairing were $\gamma|_{\Gal(K|k)\times N}=1$ and $\gamma|_{N\times N}=\Alt(\sigma)$,
\item $N$ and $S$ are normal subgroups of $G$,
\item   $[(\sigma,\gamma)]=[(\sigma^g,\gamma^g)]$ for all $g\in G/S$.
\end{enumerate}
\end{defin}
\begin{obs}
In the information  of  Galois datum we have the exact sequence of groups 
\begin{equation}\label{ecuacion:suc exact}
1\to N \to S \to \Gal(K|k)\to 1.
\end{equation}
Thus, the meaning of  condition $S=N\oplus \Gal(K|k)$ in a torsor datum is just the choice of a particular  splitting of the sequence \eqref{ecuacion:suc exact}. 
\end{obs}

\begin{lem}\label{Lema:descrip automorfismo simple Galois }
Let $B$ be a simple $S$-Galois algebra with  Galois data
$(S,K,N,\sigma,\gamma)$. Then the group $\Aut_S(B)$ is
isomorphic to
\begin{align*}
\{( \eta,\omega)\in  C^1(N, K^*)\times \mathcal{Z}(\Gal(K|k))&|
\ \partial(\eta)= \Big(\frac{\omega(\sigma)}{\sigma}, \frac{\omega(\gamma)}{\gamma}\Big), %\  \text{that is, } \\
%\omega(\sigma(x,y))\eta(xy) &= \eta(x)\eta(y)\sigma(x,y),\\
%    \omega(\gamma(s,x))\eta(\-^sx) &=
%    \bar{s}(\eta(x))\gamma(s,x),\\ &\ \ \  \text{for all $x,y\in N$, $s\in S$}
 \}
\end{align*}with product $(\eta,\omega)(\eta',\omega')= (
\eta({}^\omega\eta'),\omega\omega')$.
\end{lem}
\begin{proof}
See \cite[Proposition 5.6]{Manuel}.
\end{proof}

\begin{lem}\label{prop condiciones algebra simple}
Let $B$ be a simple $S$-Galois algebra with  Galois data
$(S,K,N,\sigma,\gamma)$. Then  $|\Aut_S(B)|=|S|$ if and only if the following conditions are satisfied
\begin{enumerate}
  \item $N$ is abelian and $K$ contains a primitive root of unity of
  order equal to the exponent of $N$.
  \item $\Hom_{\Gal(K|k)}(N,K^*)=\Hom(N,K^*)$, where
\begin{align*}
\Hom_{\Gal(K|k)}(N,K^*)=:\{f:N \to K^*:   f(n n')&=f(n)f(n'), \\
 f(q nq^{-1}) &=q f(n), \  \forall n, n' \in N, q\in \Gal(K|k)\cong G/N\}
\end{align*}

  \item $\Gal(K|k)$ is abelian and for all $\omega \in \Gal(K|k)$
  $$[(\sigma,\eta)]= [(\omega(\sigma),\omega(\eta))],$$ as elements in $H^2_{\Gal(K|k)}(N,K^*)$.
\end{enumerate}
\end{lem}
\begin{proof}
The group homomorphism 
\begin{align*}
\pi_2: \Aut_S(B) &\to \mathcal{Z}(\text{Gal}(k|K))\\
(\eta,\omega)&\mapsto \omega,
\end{align*}
induces the exact sequence $$1\to \Hom_{\Gal(K|k)}(N,K^*)\to \Aut_S(B)\overset{\pi_2}{\to} \mathcal{Z}(\text{Gal}(k|K)).$$
Therefore,
\begin{align*}
|\Aut_S(B)| &\leq |\Hom_{\Gal(K|k)}(N,K^*)| \ |\text{Im}(\pi_2)| \\
&\leq |\Hom(N,K^*)|\ |\mathcal{Z}(\text{Gal}(k|K))|\\
 &\leq |N|\ |\text{Gal}(k|K)|\\
 &= |N|\ |S/N|=|S|.
\end{align*}

Thus, $|\Aut_S(B)|=|S|$   if and only if $\text{Im}(\pi_2)=\mathcal{Z}(\text{Gal}(k|K))$,  $\mathcal{Z}(\text{Gal}(k|K))=\text{Gal}(k|K) $, $\Hom_{\Gal(K|k)}(N,K^*)=\Hom(N,K^*)$  and $|N|=|\Hom(N,K^*)|$, and these conditions are precisaly $(1)$, $(2)$ and $(3)$.

\end{proof}

\begin{lem}\label{Lema1:main teor 2}
Let $B$ be a simple $S$-Galois algebra with  Galois data
$(S,K,N,\sigma,\gamma)$. If $|\Aut_S(B)|=|S|$, then 
\begin{itemize}
\item $k$ has a primitive root of unity of order equal to the exponent of $N$,
\item there is a canonical decomposition  $S=N\oplus \Gal(K|k)$.
\end{itemize}
\end{lem}
\begin{proof}
Let $\omega\in \Gal(K|k)$. Then, there is $\eta_\omega:N\to K^*$ such that $\omega(\sigma)=\delta(\eta_\omega)\sigma$. Therefore, $\Alt(\sigma)=\Alt(\omega(\sigma))=\omega(\Alt(\sigma))$ for all $\omega\in \Gal(K|k)$. Since $N$ has exponent $n$ and $\Alt(\sigma)$ is a non-degenerate bicharacter there are  $x,y\in N$ such that $q=\Alt(\sigma)(x,y)$ is a primitive root of unity of order the exponent of $N$. Therefore, $\omega(q)=q$ for all $\omega\in \Gal(K|k)$, so $q\in k$.

Since $k$ has primitive root of unity $\Gal(K|k)$ acts trivially on $\Hom(N,K^*)=\Hom(N,k^*)$. Therefore,  by Proposition  \ref{prop condiciones algebra simple},  $S= \Aut_{\Aut_S(B)}(B) = N\oplus \Gal(K|k)$.
\end{proof}

\begin{lem}\label{Lema2:main teor 2}
Let $B$ be a simple $S$-Galois algebra with  Galois data
$(S,K,N,\sigma,\gamma)$ where $S=N\oplus \Gal(K|k)$. The Galois data is equivalent to a data $(S,K,N,\sigma',\gamma')$, where $\gamma'|_{\Gal(K|k)\times N}=1$,  $\sigma'\in Z^2(N,k^*)$ and $\gamma'|_{N\times N}=\Alt(\sigma')=\Alt(\sigma)$.
\end{lem}
\begin{proof}
Since $S$ is abelian the equation \eqref{C3} implies that for every $x\in N$, $\gamma(-,x):\Gal(K|k)\to K^*$ is a 1-cocycle. By Hilbert's Theorem 90 $H^1(\Gal(K|k),K^*)=0$, therefore there exists $\tau:N\to K^*$ such that $\gamma(g,x)=g(\tau(x))/\tau(x)$ for all $x \in N, g\in \Gal(K|k)$. Therefore, if $\sigma'=\sigma\delta(\tau)$ and $\gamma'=\gamma\partial(\tau)$, it follows that  $(S,K,N,\sigma',\gamma')$ is a new Galois datum where $\gamma'|_{\Gal(K|k)\times N}=1$, thus $\sigma'\in Z^2(N,k^*)$ and $\gamma'|_{N\times N}=\Alt(\sigma')=\Alt(\sigma)$. By the equation \eqref{C2}, $g\sigma'(x,y)=\sigma'(x,y)$ for all $g\in \Gal(K|k)$, then $\sigma'(x,y)\in k^*$ for all $x,y\in N$.
\end{proof}
\begin{teor}\label{main result 2} Let $G$ be a finite group and let $k$ be a
field.

\begin{enumerate}[i)]
  \item If $(S,K,N,\s,\gamma)$ is a torsor datum associated  to $G$, then
 $$|\Aut_G(\Ind_S^G(A(K_\sigma[N],\gamma)))|=|G|.$$
  \item Let $A$ be a $G$-Galois algebra such that $|\Aut_G(A)|=|G|$. There exists a torsor datum $(S,K,N,\s,\gamma)$ such that $A\simeq \Ind_S^G(A(K_\sigma[N],\gamma))$.
\end{enumerate} \end{teor}
\begin{proof}
Let $A$ be a $G$-Galois algebra such that $|\Aut_G(A)|=|A|$. By  Lemma \ref{Lema1:main teor 2} and Lemma \ref{Lema2:main teor 2}  $A$ has a Galois datum  $(S,K,N,\s,\gamma)$  which satisfies the first five conditions of Definition \ref{defin:torsor data}. 

By  Corollary \ref{corol: condiciones igual orden}, if  $|\Aut_G(A)|=|G|$, then $S$ is a normal subgroup of $G$. We want to find conditions that imply $A(K_\sigma[N],\gamma)\cong A(K_\sigma[N],\gamma)^{(\sigma)}$ for all $\sigma\in G$.
The Galois datum of the simple $S$-algebra  $A(K_\sigma[N],\gamma)^{(g)}$ is $$(g^{-1}Sg,g^{-1}Ng,\sigma^{(g)},\gamma^{(g)}),$$
where $$\sigma^{(g)}(x,y):=\sigma(gxg^{-1},gyg^{-1}),\  \  \gamma^{(g)}(h,x):=\gamma(ghg^{-1},gxg^{-1})$$ for all $h\in g^{-1}Sg, x,y \in g^{-1}Sg$.
By  \cite[Proposition 5.6]{Manuel} $A(K_\sigma[N],\gamma)\cong A(K_\sigma[N],\gamma)^{(g)}$ for all $g\in G$ if and only $N$ is normal in $G$ and  there is $\omega \in \Gal(K|k)$ such that $(\omega(\sigma),\omega(\eta))$ is cohomologous to $(\sigma^g,\eta^g)$. By condition (3) of Lemma \ref{prop condiciones algebra simple} $(\omega(\sigma),\omega(\eta))$   is cohomologous to $(\sigma,\eta)$, so $A(K_\sigma[N],\gamma)\cong A(K_\sigma[N],\gamma)^{(\sigma)}$ if and only if for all $g\in G$, $$[(\sigma,\eta)]=[(\sigma^g,\eta^g)],$$ as elements in $H^2_{\Gal(K|k)}(N,K^*)$, where $\sigma^g(x,y)=\sigma(gxg^{-1},gyg^{-1})$ and $\eta^g(s,x)=\eta(gsg^{-1},gxg^{-1})$, for all $x,y\in N, g\in G$. Therefore, the new Galois  datum constructed using Lemma \ref{Lema2:main teor 2} is a torsor datum.

Conversely, if $(S,K,N,\s,\gamma)$ is a torsor datum, then the first five conditions of Definition \ref{defin:torsor data} and Lemma \ref{prop condiciones algebra simple} imply that $|\Aut_S(A(K_\sigma,\gamma))|=|S|$. The conditions (6) and (7) of Definition \ref{defin:torsor data}  and Corollary \ref{corol: condiciones igual orden} imply that $|\Aut_G(\Ind_S^G(A(K_\sigma[N],\gamma)))|=|G|.$
\end{proof}

The next corollary is a generalization of \cite[Theorem 1.3]{isocategorical} and \cite[Corollary 6.2]{Davydov}.
\begin{corol}
Let $G_1$ and $G_2$ be finite groups and let $k$ be a field. Then, $G_1$ and $G_2$ are isocategorical over $k$ if and if there is a torsor datum of $G_1$ with $G_1$-Galois algebra associated $A$ such that $\Aut_{G_1}(A)\cong G_2$.
\end{corol}
\qed
\subsection{Isocategorical groups over formally real fields}

A field $k$ is called formally real if $-1$ is not a sum of squares. A formally real field with no formally real proper algebraic extensions is called a real closed field. Examples of real fields are $\mathbb{Q}$ and $\mathbb{R}$. The field of real numbers is a real closed.

Every formally real field $k$ has characteristic zero and the field extension $k\subset k(i)$ (where $i^2=-1$) is a Galois extension with $\Gal(k(i)|k)\cong \mathbb{Z}/2\mathbb{Z}$.
\begin{defin}\label{defin: real torsor data}
Let $G$ be a finite group and $k$ be a formally real field. 
\begin{enumerate}
\item  A real torsor datum for $G$ is a pair $(N,\sigma)$ where

\begin{itemize}
\item $N$ is a normal abelian elementary 2-subgroup of $G$.
\item $\sigma \in Z^2(N,\mu_2)$ is a non-degenerate 2-cocycle, where $\mu_2=\langle -1 \rangle$. 
\item  $[\sigma]=[\sigma^g]$ for all $g\in G/N$, as elements in $H^2(N,\mu_2)$.
\end{itemize}

\item A semi-real torsor datum for $G$ is a torsor datum  $(S,k(i), N,\sigma,\gamma)$, where 
\begin{itemize}
\item $N$ is an abelian 2-group 
\item $\sigma \in Z^2(N,\mu_2)$, where $\mu_2=\langle -1 \rangle$. 
\item   $[(\sigma,\gamma)]=[(\sigma^g,\gamma^g)]$  for all $g\in G/S$, as elements in $H^2_S(N,\mu_4)$, where $\mu_4=\langle i \rangle$.
\end{itemize}
\end{enumerate}
\end{defin}
\begin{obs}
A real torsor datum $(N,\sigma)$ defines a torsor datum $(N,k,N,\sigma,\gamma_\sigma)$, where $\gamma_\sigma(x,y)=\sigma(x,y)\sigma(xy,x^{-1})\sigma(x,x^{-1})^{-1},$ for all $x, y\in V$.
\end{obs}

\begin{lem}\label{Lema:resultado3}
Let $N$ be an elementary abelian 2-group, $k$ a real closed field and let $S=N\oplus\Gal(k(i)|k)$.  Every 2-cocycle in $Z^2_S(N,k(i)^*)$ is cohomologous to a 2-cocycle in  $\sigma\in Z^2_S(N,\mu_2)$ and  $\alpha, \beta  \in Z^2_S(N,k(i)^*)$ are cohomologous as elements in $Z^2_S(N,k(i)^*)$ if and only if $\alpha, \beta$ are cohomologous as elements in $Z^2_S(N,\mu_4)$.
\end{lem}
\begin{proof}
Since $k$ is a real closed field,  $k(i)$ is an algebraically closed field of characteristic zero. It is well known that for finite abelian groups every 2-cocycle with values in $k(i)^*$ is cohomologous to a bicharacter \cite[Proposition 2.6]{Tambara-Func}, so we can suppose that $\sigma$ is a bicharacter  and since $N$ is an elementary 2-group $\sigma$ takes values in $\mu_2$.   Then, the equation \eqref{C2} implies that $\gamma(g, -):N\to k(i)^*$ is a character for all $g \in S$. Again, since $N$ is an elementary abelian 2-group,  $\gamma$ only takes values in $\mu_2$. Now, equation \eqref{C3} implies that $\gamma:S\times N\to \mu_2$ in fact is a pairing and therefore every 2-cocycle in $Z^2_S(N,k(i)^*)$ is cohomologous to a 2-cocycle in $Z^2_S(N,\mu_2)\subset Z^2_S(N,\mu_4)$. Let $(\sigma,\gamma)\in Z^2_S(N,\mu_2)\subset Z^2_S(N,\mu_4)$ and $\eta:N\to \mu_\infty$, such that $\partial(\eta)=(\sigma,\gamma)$. Let $g\in \Gal(K(i)|k)\subset S$ be the  conjugation  on $k(i)$. Then  $\gamma(g,x)=\overline{\eta(x)}/\eta(x)$ for all $x\in N$. Since  $\gamma(g,x)\in \mu_2$, then $\eta(x)\in \mu_4$ for all $x\in N$. Then, if $\alpha, \beta \in Z^2_S(N,\mu_2)\subset Z^2_S(N,\mu_4)$ and they are cohomologous as elements in $Z^2_S(N,k(i)^*)$ they are cohomologous as elements in $Z^2_S(N,\mu_4)$.

\end{proof}

\begin{teor}\label{main result 3} Let $G$ be a finite group and $k$ be a
formally real closed field. Let $A$ be a Galois algebra over $k$ such that $|\Aut_G(A)|=|G|$. Then, $A\simeq \Ind_S^G(A(k(i)_\sigma[N],\gamma))$ or $A\simeq \Ind_S^G(A(k_\sigma[N],\gamma))$ for a semi-real  or real torsor datum, respectively.
\end{teor}
\begin{proof}
Since $k$ is a formally real field the unique primitive root of unity is $-1$. Let $A$ be a $G$-Galois algebra over $k$ such that $|\Aut_G(A)|=|G|$.   By Theorem \ref{main result 2} there exists a torsor datum $(S,K,N,\s,\gamma)$, where $N$ is an abelian group of exponent $2$. Therefore $N$ is an elementary abelian group. 

Since $H^2(N,k^*)\cong H^2(N,\mu_2)$ for any formally real closed field, if $K=k$  the torsor datum defines a real torsor. Now, if $K$ is a proper extension of $k$, then $K=k(i)$ and by Lemma \ref{Lema:resultado3} the torsor datum is equivalent to a semi-real torsor for $G$.
\end{proof}

\begin{corol}\label{Corol Clasification over Q nad Z[1/2]}
Let $G_1$ and $G_2$ be finite groups and $k$ be a formally real field. Then:

\begin{enumerate}
\item $G_1$ and $G_2$ are isocategorical over $k$  if and only if  there is a real or semi-real torsor datum such that $\Aut_{G_1}(A)\cong G_2$, where $A$ is the associated Galois algebra to the datum.

\item $G_1$ and $G_2$ are isocategorical over $k$ if and if they are isocategorical over $\mathbb{Q}$.
\end{enumerate}
\end{corol}
\begin{proof}
Let $k'$ be a formally real closed extension of $k$.  If $G_1$ and $G_2$ are isocategorical over $k$, then by Proposition \ref{Prop equivalencias con isocat  existencia bigalois} there exists a $G_1$-Galois algebra $A$ such that $\Aut_{G_1}(A)\cong G_2$. The $G_1$-algebra $A'=A\otimes_k k'$ is also Galois and  $\Aut_{G_1}(A')\cong G_2$. By Theorem \ref{main result 3} there exists a real or  semi-real torsor datum such that $A' \simeq \Ind_S^{G_1}(A(k'(i)_\sigma[N],\gamma))$. By the definition of real of semi-real datum, we have that $A'$ has a  $k$-form $D:=\Ind_S^{G_1}(A(k(i)_\sigma[N],\gamma))$ and by Lemma \ref{Lema:resultado3},  $\Aut_{G_1}(D)\cong \Aut_{G_1}(A')\cong G_2$. Therefore $G_1$ and $G_2$ are isocategorical over $k$. For the second part note that also $A'$ has a rational form $L:=\Ind_S^{G_1}(A(\mathbb{Q}(i)_\sigma[N],\gamma))$ and $\Aut_{G_1}(L)\cong \Aut_{G_1}(A')\cong G_2$.
\end{proof}

\section{Galois algebras and  Weil representations}\label{Section:Galois algebras and  Weil representations}

\subsection{ Categorical setting }

Let $\G$ be a small groupoid. We will denote by $\underline{\Aut(\G)}$ the
monoidal groupoid where objects are autoequivalences of $\G$,
morphisms are natural isomorphism and the monoidal structure is given by the
composition of functors.

Given a group $X$, we will denote by $\underline{X}$ the discrete
monoidal category where objects are elements of $X$ and the monoidal structure is given by the product of $X$.

A normalized (right) \textit{action} of a group $X$ on a category $\G$ is a monoidal functor $\rho_*:\underline{X^{op}}\to  \underline{\Aut(\G)}$, where  $X^{op}$ denoted the opposite group of $X$. More concretely, an action of $X$ on $\G$ consists of the following data:

\begin{itemize}
\item functors $(-)^{(x)}:\G\to \G, A\mapsto A^{(x)}$ for all $x\in X$,
\item natural isomorphisms $\gamma_{x,y}^A: A^{(xy)}\to A^{(x)(y)}$ for 	all $x,y\in X, A\in \G$,
\end{itemize} such that $(-)^{1}=\id_\G$, $\gamma_{x,1}^A=\gamma_{1,x}^A=\id_{A^{(x)}}$ for all $x\in X$ and the  diagrams
\begin{equation}\label{2-cociclo G-accion}
\begin{diagram}
\node{ A^{(xyz)}}\arrow{s,l}{ \gamma_{x,yz}^{A}}\arrow{e,t}{\gamma_{xy,z}^A}\node{A^{(xy)(z)} }\arrow{s,r}{(\gamma_{x,y}^{A})^{(z)}}\\
\node{ A^{(x)(yz)}}\arrow{e,t}{\gamma_{y,z}^{A^{(x)}}}
\node{A^{(x)(y)(z)}}
\end{diagram}
\end{equation}commute for all $x,y,z\in G, A\in \G$.

There is two different groupoids associated to a
normalized action of a group $X$ on a groupoid $\G$: The groupoid
of equivariant objects $\G^X$ and the quotient groupoid $\G/\!\!/X$.

\subsubsection{The groupoid of $X$-equivariant objects}
The groupoid $\G^X$ of $X$-equivariant objects  is defined as follows: an object in $\G^X$ is
a pair $(A,u)$, where $A\in \G$ is an object and $u_x: A^{(x)}\to
A$ is a family of morphisms such that $u_1=\id_A$ and the diagrams
\begin{equation}\label{diagrama u_g}
\begin{diagram}
\node{ A^{(xy)}}\arrow{s,l}{ \gamma_{x,y}}\arrow{e,t}{u_{xy}}\node{A }\\
\node{ A^{(x)(y)}}\arrow{e,t}{(u_x)^{(y)}}
\node{A^{(y)}}\arrow{n,r}{u_y}
\end{diagram}
\end{equation} commute for all $x,y\in X$. A morphism from $(A,u)$ to $(A',u')$ is
a morphism $f\in \Hom_\G(A,A')$ such that the diagrams

$$
\begin{diagram}
\node{ A^{(x)}}\arrow{s,l}{f^{(x)}}\arrow{e,t}{u_{x}}\node{A }\arrow{s,r}{f}\\
\node{ A'^{(x)}}\arrow{e,t}{u'_x} \node{A}
\end{diagram}
$$commute for all $x\in X$.

\subsubsection{The quotient groupoid of a $X$-action}
The groupoid $\G/\!\!/X$ has objects Obj$(\G)$ and morphisms
$$\Hom_{\G/\!\!/X}(A,B)=\{(g,\alpha)| g\in X \text{ and } \alpha \in
\Hom_\G(A^{(g)},B)\}.$$ The composition of
$(g,\alpha_g)\in\Hom_{\G/\!\!/X}(A,B)$ and  $(h,\alpha_h)\in
\Hom_{\G/\!\!/X}(B,C)$  is defined by $ (g,\alpha_g)\circ
(h,\alpha_h):=(gh,\alpha_g\odot \alpha_h)\in \Hom_{\G/\!\!/X}(A,C)$
where $\alpha_g\odot \alpha_h: A^{(gh)}\to C$ is defined by the
commutativity of the diagram
$$
\begin{diagram}
\node{ A^{(gh)}}\arrow{s,l}{ \gamma_{g,h}}\arrow{e,t}{\alpha_g\odot \alpha_h}\node{C }\\
\node{ A^{(g)(h)}}\arrow{e,t}{\alpha_g^{(h)}}
\node{B^{(h)}}\arrow{n,r}{\alpha_h}
\end{diagram}
$$ The identity of an object $A$ is given by $(1,\id_A)$ and the
associativity of the composition follows from the naturality of $\gamma$ and the commutativity of the diagrams \eqref{2-cociclo G-accion}:
\begin{align*}
\alpha_g \odot (\alpha_h\odot \alpha_k) &=  \alpha_k\alpha_h^{(k)} \gamma_{h,k}^A\alpha_g^{(hk)}\gamma_{g,hk}^A\\
=&\alpha_k\alpha_h^{(k)}( \gamma_{h,k}^A\alpha_g^{(hk)})\gamma_{g,hk}^A\\
&=  \alpha_k\alpha_h^{(k)}(\alpha_g^{(h)(k)}\gamma_{h,k}^{A^{(g)}})\gamma_{g,hk}^A\\
&=  \alpha_k\alpha_h^{(k)}\alpha_g^{(h)(k)}(\gamma_{h,k}^{A^{(g)}}\gamma_{g,hk}^A)\\
&=  \alpha_k\alpha_h^{(k)}\alpha_g^{(h)(k)}((\gamma_{g,h}^A)^{(k)}\gamma_{gh,k}^A)\\
&= \alpha_k (\alpha_h\alpha_g^{(h)}\gamma_{g,h}^A)^{(k)}\gamma_{gh,k}^A\\
&= \alpha_k (\alpha_g\odot \alpha_h)^{(k)}\gamma_{gh,k}^A\\
&= (\alpha_g\odot\alpha_h )\odot \alpha_k.
\end{align*}

The next proposition gives a complete description of the structure
of the groupoids $\G^X$ and $\G/\!\!/X$.

\begin{prop}\label{Prop estructura grupoide accion  e invariante}
Let $X$ be a group acting on a groupoid $\G$. Then
\begin{itemize}
\item[(a)] The set of isomorphism classes of $\G/\!\!/X$ is the set
of orbits of the $X$-set of isomorphism classes of objects of
$\G$.

\item[(b)] For every $A\in \text{Obj}(\G)$, the sequence
\begin{align}\label{sequence exact of quotient groupoid}
1\to \Aut_\G(A)\to \Aut_{\G/\!\!/X}(A)\to St([A])\to 1
\end{align} is exact. Where $\Aut_\G(A)\to \Aut_{\G/\!\!/X}(A), f\mapsto (1,f)$,
$\Aut_{\G/\!\!/X}(A)\to St([A]), (g,\alpha_g)\mapsto g$ and
$St([A])=\{g\in X| A^{(g)}\cong A\}$.

\item[(c)] Suppose that $St([A])=X$.  There is a bijective
correspondence between splittings of the exact sequence
\eqref{sequence exact of quotient groupoid} and families of morphisms
$\{u_g:A^{(g)}\to A\}_{g\in X}$ such that $(A,u_g)\in \G^{X}$.

\item[(d)] If $(A,u_g)\in \G^{X}$, then $X$ acts by the right on
$\Aut_\G(A)$ by $$f\cdot g:=u_g(
f)^{(g)} u_g^{-1},$$for all $g\in X, f\in \Aut_\G(A)$. Moreover,
$\Aut_{\G^X}((A,u_g))=Aut_\G(A)^{X}$ and the isomorphism classes
of objects in $\G^X$ with underling object $A$ are in
correspondence with elements of the pointed set
$H^1(X,\Aut_\G(A))$.
\end{itemize}
\end{prop}
\begin{proof}

The parts (a), (b) and (c) are intermediate from the definitions.

Let $(A,u_g)\in \G^X$. The group morphism  $\hat{u} :X\to \Aut_{\G/\!\!/X}(A), g\mapsto (g,u_g)$ is a splitting of the exact sequence \eqref{sequence exact of quotient groupoid}. Therefore, the group $X$ acts (on the right) on $\Aut_\G(X)$ by $$f\cdot g:=u_{g^{-1}}\odot f\odot u_{g}$$for all $g\in G, f\in \Aut_\G(X)$. Using the diagrams \eqref{diagrama u_g}, we have $\id_A=u_gu_{g^{-1}}^{(g)}\gamma_{g^{-1},g}^A$. Hence
\begin{align*}
f\cdot g &:= u_{g^{-1}}\odot f\odot u_{g}\\
 &= u_g f^{(g)}u_{g^{-1}}^{(g)}\gamma_{g^{-1},g}^A\\ 
 &= u_gf^{(g)} u_g^{-1},
\end{align*}for all for all $g\in G, f\in \Aut_\G(A)$. Now is clear that $\Aut_{\G^X}((A,u_g))=Aut_\G(A)^{X}$. 

The pointed set $H^1(X,\Aut_\G(A))$ is in bijective correspondence with the set of conjugacy classes of splittings of the exact sequence \eqref{sequence exact of quotient groupoid}. Therefore,  by (c) and the first part of (d), $H^1(X,\Aut_\G(A))$ is in bijective correspondence with the set of isomorphism classes of objects in $\G^X$.
\end{proof}

\subsection{Actions on the groupoid of $N$-Galois algebras}

\subsubsection{Group extensions}

Let $X$ and $N$ be a groups. An $X$-crossed system over $N$ is a pair of maps 
\begin{itemize}
\item $\cdot :X\times N\to N, (n,x)\mapsto \ ^xn$
\item $\theta:X\times X\to N$
\end{itemize}such that the set $N\times X$ with the product given by $$(n,x)(m,y):=(n(^xm)\theta(x,y),xy), \  \  \text{for all } (n,x), (m,y)\in N\times X$$ is a group with unit $(1,1)$. It is easy to see that a pair $(\cdot, \theta )$ is a $X$-crossed system if and only if 
\begin{itemize}
\item $^x(nm)=\ (^xn)(^ym)$
\item $\theta(x,y)\ ^{xy}n=\ ^x(^yn)\theta(x,y)$
\item $\theta(x,y)\theta(xy,z)=\ ^x\theta(y,z)\theta(x,yz)$
\item $\theta(x,1)=\theta(1,x)=1$
\item $^1n=n$
\end{itemize}for all $x,y,z\in X, m,n\in N$.

We will denote by $ N\#_\theta X$ the group  and will call it a crossed product of $X$ by $N$. Note that we have an exact sequence $$1\to N\to N\#_\theta X\to X\to 1,$$ where $N\to N\#_\theta X, n\mapsto (n,1)$ and $N\#_\theta X\to X, (n,x)\mapsto x$. 

Let $N$ be a finite group and let $k$ be a field. It is well known that every morphism between Galois algebras over $k$ is an isomorphism. Therefore, the category $\mathcal{G}al(N,k)$ of $N$-Galois algebras over a $k$ is a groupoid.

Let $X$ be a finite group and let $(N,\cdot,\theta)$ be an $X$-crossed system. The crossed system defines an action of $X$ on $\mathcal{G}al(N,k)$ as follows. Let $A$ be a $N$-Galois algebra over $k$. Let $A^{(x)}=A$ as algebras with  new $N$-action given by $n\cdot a:=\ ^{x}n\cdot a$, for all $x \in X, n\in N, a\in A$ and  let
\begin{align*}
\gamma_{x,y}: A^{(xy)} &\to\ A^{(x)(y)} \\
a &\mapsto \theta(x,y)a
\end{align*}
for all $a\in A$. It is straightforward to see that the above formulas define a normalized right  $X$-action on $\mathcal{G}al(N,k)$.
\begin{prop}\label{Prop:Isomorfismo Automorfismo Galois y auto grupoide equivariant} Let $X$ be a finite group, $(N,\cdot,\theta)$ be an $X$-crossed system. Let $A\in \mathcal{G}al(N,k)$. Then, there is an isomorphism of exact sequences 
$$
\begin{diagram}
\node{1}\arrow{e,t}{}\node{ \Aut_N(A)}\arrow{s,l}{ \id}\arrow{e,t}{}\node{ \Aut_{N\#_\theta X}(\Ind_N^{N\#_\theta X}(A))} \arrow{s,l}{ \cong}\arrow{e,t}{}\node{X}\arrow{s,l}{ \id}\\
\node{1} \arrow{e,t}{} \node{ \Aut_{\mathcal{G}al(N,k)}(A)}\arrow{e,t}{}
\node{\Aut_{\mathcal{G}al(N,k)/\!\!/X}(A)}\arrow{e,t}{}\node {X}
\end{diagram}$$
where the first sequence is \eqref{primera sucesion exacta} and the second one is \eqref{sequence exact of quotient groupoid}.
\end{prop}
\begin{proof}

Let  $G:= N\#_\theta X$, the crossed product group. Let $A$ be a $N$-Galois algebra. For each $x\in X$ the map \begin{align*}
  \psi_x:\Ind_N^G(A)&\to \Ind_{N}^G(A^{(x)})\\
f&\mapsto \psi_x(f)=[(n,z)\mapsto f(^xn \theta(x,z),xz))], 
\end{align*} defines a $G$-algebra isomorphism. The map 
\begin{align*}
F:\Aut_{\mathcal{G}al(N,k)/\!\!/X}(A) &\to \Aut_{N\#_\theta X}(\Ind_N^{N\#_\theta X}(A))\\
(\alpha_x,x) &\mapsto \Ind(\alpha_x)\circ\psi_x,
\end{align*}is a morphism of groups. In fact, 
\begin{align*}
[F(\alpha_y,y)\circ F(\alpha_x,x)] (f)(n,z) &=[\Ind(\alpha_y)\psi_y \Ind(\alpha_x)\psi_x] (f) (n,z)\\
&=[\Ind(\alpha_y)\Ind(\alpha_x)\psi_x] (f) (^yn\theta(y,z),yz)\\
&=[\Ind(\alpha_y)\Ind(\alpha_x)] (f) (^x(^yn))^x\theta(y,z)\theta(x,yz),xyz)\\
&=[\Ind(\alpha_y)\Ind(\alpha_x)] (f)( ^x(^yn)^x\theta(y,z)\theta(x,yz),xyz)\\
&=[\Ind(\alpha_y)\Ind(\alpha_x)] (f)( ^x(^yn)\theta(x,y)\theta(xy,z),xyz)\\
&=[\Ind(\alpha_y)\Ind(\alpha_x)] (f)( \theta(x,y)(^{xy}n)\theta(xy,z),xyz)\\
&=[\Ind(\alpha_y)\Ind(\alpha_x)\Ind(\gamma_{x,y})] (f)( ^{xy}n\theta(xy,z),xyz)\\
&=[\Ind(\alpha_y)\Ind(\alpha_x)\Ind(\gamma_{x,y})\psi_{xy}] (f)( n,z)\\
&=[\Ind(\alpha_y\alpha_x)\Ind(\gamma_{x,y})\psi_{xy}] (f)( n,z)\\
&=[\Ind(\alpha_y\alpha_x\gamma_{x,y})\psi_{xy}] (f)( n,z)\\
&= F(\alpha_x\odot \alpha_y,xy)(f)(n,z)
\end{align*}for all $(n,z)\in N\#_\theta X$.

It follows from the proof of Proposition \ref{Prop: exct sequence Aut_G} that $F$ is injective.  We claim that $F$ is surjective.  Let $W\in \Aut_{N\#_\theta X}(\Ind_N^{N\#_\theta X}(A))$. There is a unique $x\in X$ such that $W(\chi_1^x)=\chi_1^1$,  so $W\psi_x^{-1}(\chi_1^y)=\chi_1^y$ for all $y\in X$. If $\alpha_x=(W\psi_x^{-1})|_A$, then $W\psi_x^{-1}=\Ind(\alpha_x)$, thus $W=F(\alpha_x,x)$.
\end{proof}

\subsection{The Finite  Weil representation associated to a simple Galois algebra}

Let $S$ be a  finite abelian group and $A$ be an  $S$-Galois algebra. Then,  $\Aut_S(A)$ is an abelian group. Let   $$\operatorname{St}(A)= \{g\in \Aut(S)| A^{(g)}\cong A \text{ as $S$-algebras}\}.$$   
\begin{prop-def}\label{weil action}
Let $S$ be a  finite abelian group and $A$ be an  $S$-Galois algebra. If we choose an isomorphism of $S$-algebras $\alpha_g:A^{(g)}\to A$ for each $g\in \St(A)$, then
\begin{align*}
g\cdot f &:= \alpha_g^{-1} f \alpha_{g},\\
\theta(x,y) &:= \alpha_x^{-1}\alpha_y^{-1}\alpha_{xy} \in \Aut_S(A).
\end{align*}define a crossed system of $\operatorname{St}(A)$ over $\Aut_S(A)$. The crossed product  $\Aut_S(A)\#_\theta \operatorname{St}(A)$ acts on $A$ by algebra automorphisms  as  
\begin{equation}\label{ecuacion de la accion de Weil}
(\psi,x)\cdot a=\psi(\alpha_{x}^{-1}(a)).
\end{equation}This action  will be called the \textit{Weil action}.  If the cohomology class  $$\theta \in Z^2(\St(A),\Aut_S(S))$$ is zero,  $\St(A)$ acts on $A$  and this action will be also called   the \textit{Weil action}.
\end{prop-def}
\begin{proof}
Straightforward.
\end{proof}

If $A$ is a simple algebra, once a simple $A$-module $M$ is fixed, there is a canonical isomorphism $A\cong M_n(D)$, where $D=\End_A(M)$.  Using the Skolem-Noether theorem, the Weil action defines and is defined by  a  unique (up to isomorphism) projective representation  $\rho :\Aut_S(A)\#_\theta \operatorname{St}(A)\to \operatorname{PGL}(M):= \operatorname{GL}_k(M)/k^*$ by the equation $$\alpha_g(f)(v)=\rho_gf\rho_g^{-1}(v)$$ for all $f\in M_n(D), v\in M$. 

\begin{ejem}\label{cannoical weil}
Let $V$ be an abelian group of odd order and $\omega \in \operatorname{Hom}(\wedge^2 V,k^*)$ a non-degenerate skew-symmetric bicharacter. Let $A=k_\omega[V]$ be the twisted group algebra.  Then $A$ is a $V$-Galois algebra,  \begin{equation}\label{definicion de grupo simplectico}
\Sp(V,\omega)=\{g\in \Aut(V)| \omega(x,y)=\omega(g(x),g(y))\}= \St(A)
\end{equation} and $\Aut_V(A)=\widehat{V}$. The group $\Sp(V,\omega)$ acts on $A$ by $\alpha_g:A\to A, u_x\mapsto u_{g(x)}$. Fixing a Lagrangian decomposition of $V=U\oplus W$,  $A$ acts on $M:=\text{Span} (\{t_b| b\in U\})$ by $u_{x\oplus y}\cdot t_b=\omega(y,b)t_{x+b}$ and the associated representation corresponds to the usual Weil representation.
\end{ejem}

\begin{obs}
The Weil action defined in Example \ref{cannoical weil} can be seen as the \textit{canonical Weil representation} because it is defined only in terms of the pair $(V,\omega)$. 
\end{obs}

\subsection{The Weil representation of a  Symplectic module}

Let $V$ be an abelian group. A  skew-symmetric form on $V$ is a bicharacter $\omega: V\times V\to \mathbb{C}^*$ such that $\omega(v,v)=1$ for all  $v\in V$. We will denote by $\bigwedge^2 \widehat{V}$ the abelian group of all skew-symmetric forms on $V$.

A symplectic module  is a pair $(V,\omega)$, where $V$ is a finite abelian group and $\omega$ is a non-degenerate skew-symmetric form. We define $\Sp(V,\omega)$ the symplectic group of $(V,\omega)$ by \eqref{definicion de grupo simplectico}.% $$\Sp(V,\omega)=\{g\in \Aut(V)| \omega(g(x),g(y))=\omega(x,y), \forall x,y\in V\}.$$
\begin{ejem}\label{Ejem:smplectico lineal a simplecto modulo}
Let $\mathbb{F}_{p^n}$ be the finite field with $p^n$ elements. Let $V$ be an $\mathbb{F}_{p^n}$-vector space and  $\langle-,-\rangle:V\times V\to \mathbb{F}_{p^n}$  a  symplectic $\mathbb{F}_{p^n}$-bilinear form. Using the trace  map $\operatorname{Tr}: \mathbb{F}_{p^n}\to \mathbb{F}_{p}, x\mapsto x+x^p+\cdots x^{p-1}$ we define the skew-symmetric form
\begin{align*}
\omega: V\times V&\to \mathbb{C}^*\\
(v,w)&\mapsto e^{\frac{ 2\pi i\operatorname{Tr}(\langle v,w\rangle)}{p} }.
\end{align*}
\begin{prop}\label{Prop:ejem simplec lineal simplec module}
The pair  $(V,\omega)$ is a symplectic module and the linear symplectic group $\Sp_k(V)$ is a subgroup of $\Sp(V,\omega)$.
\end{prop}
\begin{proof}
It is clear that $\omega \in \bigwedge^2 \widehat{V}$, so we only will see that $\omega$ is non-degenerate. Suppose that $\omega(v,w)= 1$ for all $w\in V$, so $\operatorname{Tr}(\langle v,w\rangle)=0$  for all $y\in V$. Therefore,  $\operatorname{Tr}(c\langle v,w\rangle)=\operatorname{Tr}(\langle v,cw\rangle)=0$ for all $c\in \mathbb{F}_{p^n}$. Since the bilinear form $\mathbb{F}_{p^n}\times \mathbb{F}_{p^n}\to \mathbb{F}_{p}, (a,b)\mapsto \operatorname{Tr}(ab)$ is non-degenerate, then $\langle v,w\rangle=0$ for all $w\in V$. Finally, since $\langle-,-\rangle$ is non-degenerate $v=0$, so $\omega$ is a non-degenerate.
\end{proof}
\end{ejem}

Let $(V,\omega)$ be a symplectic module. By Proposition \ref{Prop:exis de 2-cociclo para simplectco} there exists $\alpha\in Z^2(V,\mathbb{C}^*)$ such that $\omega =\operatorname{Alt}(\alpha)$  and $$\Sp(V,\omega)=\{g\in \Aut(V)| \exists \eta \in C^1(V,\mathbb{C}^*): \frac{\alpha^g}{\alpha}=\delta(\eta) \}.$$

The tuple $(V,k,V,\alpha,\eta_\alpha)$ is a torsor datum. Hence, $A=\mathbb{C}_{\alpha}[V]$ is a simple $V$-Galois algebra and  by Lemma \ref{Lema:descrip automorfismo simple Galois }  $$\Aut_V(A)= \widehat{V}.$$ 

Let us fix for every $g\in \Sp(V,\omega)$ a cochain $\eta_g\in C^1(V,\mathbb{C}^*)$ such that $\frac{\alpha^g}{\alpha}=\delta(\eta_g)$. The map $\alpha_g: A\to A^{(g)}, u_x\mapsto \eta_g(x)u_{g(x)}$ is an isomorphism of $V$-algebras. Thus, by Proposition \ref{weil action} there is an associated crossed product group  $$\ASp(V,\omega):=\widehat{V}\#_{\theta} \Sp(V,\omega).$$ Following \cite{isocategorical} we call $\ASp(V,\omega)$ the  \emph{affine pseudo-symplectic group}.

The Weil representation of the affine pseudo-symplectic group $\ASp(V,\omega)$ is the projective representation associated to the Weil action of $\ASp(V,\omega)$ on the simple algebra $A=\mathbb{C}_{\alpha}[V]$, see Proposition \ref{weil action}.

\subsection{The Weil representation associated to a  quadratic module }

\begin{defin}
An abelian group $V$ isomorphic to  $(\mathbb{Z}/n\mathbb{Z})^r$ will be called a  homogeneous module. The number  $r$ is  called the rank and $n$ the exponent of $V$. 
\end{defin}
We will denote by $\mu_n$ the group of all root of unity of order $n$ in $\mathbb{C}$.

A quadratic  form on a homogeneous module $V$ of exponent $n$  is a  map $q:V\to \mu_n$ such that  $q(x)=q(x^{-1})$  and the map 
\begin{align*}
\omega_q:V\times V &\to \mathbb{C}^*\\
(x,y)&\mapsto \frac{q(xy)}{q(x)q(y)} 
\end{align*}
is a bicharacter.  The group of all quadratic forms on a homogeneous module $V$ will be denoted by $\Quad(V)$. 
\begin{defin}
A quadratic module is a pair $(V,q)$, where $V$ is a finite homogeneous module and $q$ is a quadratic form on $V$ such that the associated skew-symmetric form $\omega_q$ is non-degenerate.
\end{defin}
Given a quadratic module $(V,q)$ the  orthogonal group  of $(V,q)$ is defined as $$O(V,q)=\{g \in \Aut(V) |  q(x)=q(g(x)), \  \forall x\in V\}.$$ 
\begin{ejem}\label{Ejemplo: modulo cuadratico de spacop cuadratico}
Let $k$ be field and $V$ a finite dimensional $k$-vector space. Recall that a (usual) quadratic form is a function $q:V\to k$ such that $q(ax+by)=a^2q(x)+abB(x,y)+ b^2q(y)$ for all $a,b \in k, x, y\in V$, where $B$ is a $k$-bilinear form on $V$. The form $q$ is called nondefective if the bilinear form $B$ is non-degenerate. Note that if $k$ has characteristic different from 2, the quadratic form is totally determined by $B$. A quadratic (linear) space is a pair $(V,q)$ where $V$ is a linear space and $q$ is a quadratic form whose associated bilinear form is non-degenerate. Given a linear quadratic module $(V,q)$ the linear orthogonal group  of $(V,q)$ is defined as $$O_k(V,q)=\{g \in \operatorname{GL}_k(V) |  q(x)=q(g(x)), \  \forall x\in V\}.$$ 

Let $V$ be an $\mathbb{F}_{p^n}$-vector space and  $q:V\to \mathbb{F}_{p^n}$  a  quadratic form. Le us define the quadratic form 
\begin{align*}
\widehat{q}: V&\to \mu_p\\
v&\mapsto e^{\frac{2\pi i\Tr(q(v))}{p}}.
\end{align*}
\begin{prop}
If $q$ is nondefective then the pair $(V,\widehat{q})$ is a quadratic module of exponent $p$ and the linear orthogonal group $O_{\mathbb{F}_{p^n}}(V,q)$ is a subgroup of $O(V,\widehat{q})$.
\end{prop}
\begin{proof}
For all $v\in V$, $$\widehat{q}(-v)=e^{\frac{2\pi i\Tr(q(-v))}{p}}= e^{\frac{2\pi i(-1)^2\Tr(q(v))}{p}}=e^{\frac{2\pi i\Tr(q(v))}{p}}=\widehat{q}(v).$$
Let $B$ be  the $\mathbb{F}_{p^n}$-bilinear form associated to $q$. Then  $$\omega_{\widehat{q}}(x,y)= \frac{\widehat{q}(x+y)}{\widehat{q}(x)\widehat{q}(y)}= e^{\frac{2\pi i\Tr(B(x,y)) }{p} }$$ for all $x,y\in V$. Thus, $\omega_{\widehat{q}}$ is a skew-symmetric form. It follows from Proposition \ref{Prop:ejem simplec lineal simplec module} that  $\omega_{\widehat{q}}$ is non-degenerate.
\end{proof}
\end{ejem}

For a homogeneous module $V$  we  denote by $\Bil(V)$ the abelian group of all bicharacter with values in $\mathbb{C}^*$. 
The map $\operatorname{Tr}:\Bil(V)\to \Quad(V), \Tr(b)(x)=b(x,x)$ defines a morphism of groups with kernel $\bigwedge^2 \widehat{V}$.

\begin{lem}\label{Lemma1:cohomologia ortogonales}
Let $V$ be a finite homogeneous module.  The sequence $$0\to \bigwedge^2 \widehat{V} \to \Bil(V) \to \Quad(V)\to 1$$ is  exact.
\end{lem}
\begin{proof}
It is clear that we only need to show that the trace map $$\Tr: \Bil(V) \to \Quad(V)$$ is surjective.   If $V=\langle g \rangle$ has rank one (that is, $V$ is a cyclic group of order $n$)  and $q\in \Quad(V)$, then  $b:V\times V\to \mu_n$, $b(g^i,g^j)= q(g)^{ij}$ is a bilinear map with $\Tr(b)=q$.  Let $V=A\oplus C$ where $C$ has rank one,  $q\in \Quad (V)$ and suppose that for $q|_A$ and $q|_C$ we have $b_A \in \Bil(V)$, $b_C\in \Bil(C)$ with $\Tr(b_A)=q|_A$, $\Tr(b_C)=q|_C$. Then  
\begin{align*}
b:V\times V &\to \mu_n,\\
a\oplus c, a'\oplus c' &\mapsto b_A(a,a')b_C(c,c')\omega_q(a,c')
\end{align*}
is a bichartacter with $\Tr(b)=q$. 
\end{proof}

Let $V$ be a finite homogeneous module of exponent $n$. We define $$C_0^1(V,\mu_n)=\{\eta\in C^1(V,\mu_n)| \delta(\eta) \text{ is bicharacter and } \eta(x^2)=\eta(x)^2,  \forall x\in V\},$$ and  the morphism
\begin{align*}
\delta:C_0^1(V,\mu_n)&\to  \bigwedge^2 \widehat{V}\\
\eta&\mapsto \delta(\eta).
\end{align*}
\begin{lem}\label{Lemma2:cohomologia ortogonales}
The sequence $$0\to  \widehat{V} \to C_0^1(V,\mu_n) \to \bigwedge^2 \widehat{V}\to 1$$ is  exact. 
\end{lem}
\begin{proof}

Let $b\in \bigwedge^2 \widehat{V}$. Then, $1=b(xy,xy)=b(x,y)b(y,x)$, which implies that $b(x,y)=b(y,x)^{-1}$ for all $x,y \in V$. 

If $V$ is cyclic $\bigwedge^2 \widehat{V}=0$, and $\widehat{V} = C_0^1(V,\mu_n)$. Let $V=A\oplus C$, where $C$ has rank one.  Suppose that there exists $\eta_A\in C_0^1(A,\mu_n)$ such that $\delta(\eta_A)=b|_{A\times A}$. Let $\widehat{\eta_A} \in C_0^1(V,\mu_n),$ by $  \widehat{\eta_A}(a\oplus c)=\eta_A(a)$. Then $b':=b\delta(\widehat{\eta_A})^{-1}\in \bigwedge^2 \widehat{V}$ and $b'|_{A\times A}=1$. If we  define $s\in C^1(V,\mu_n)$ by $s(a\oplus c)=b(a,c)^{-1}$. A simple calculation shows that $\delta(s)=b'$. Therefore, $s\in C_0^1(V,\mu_n)$. Finally, for $\gamma:=  \widehat{\eta_A}s \in C_0^1(V,\mu_n)$ we have $\delta(\gamma)=\delta(s)\delta(\widehat{\eta_A})=b$.
\end{proof}

\begin{corol}\label{Corol:existencia 2-cociclo quadratic}
Let $(V,q)$ be a quadratic module and $b\in \Bil(V)$ that satisfies $\Tr(b)=q$. Then  
$$O(V,q)=\{g\in \Aut(V)|\exists \eta \in C_0^1(V,\mu_n) : b^{g}/b=\delta(\eta)\}.$$
\end{corol}
\begin{proof}
It is clear that  $$\{g\in \Aut(V)|\exists \eta \in C_0^1(V,\mu_n) : b^{g}/b=\delta(\eta)\}\subseteq O(V,q).$$ Let $b\in \Bil(V)$ be such that $\Tr(b)=q$ and $g\in O(V,q)$. By Lemma \ref{Lemma1:cohomologia ortogonales} $b^{g}/b\in \bigwedge^2 \widehat{V}$ and by Lemma  \ref{Lemma2:cohomologia ortogonales} there exists $\eta\in C_0^1(V,\mu_n)$ such that $\delta(\eta)=b^{g}/b$. Therefore, $$O(V,q)=\{g\in \Aut(V)|\exists \eta \in C_0^1(V,\mu_n) : b^{g}/b=\delta(\eta)\}.$$
\end{proof}

Let $(V,q)$ be a quadratic module. We define the Weil representation of $O(V,q)$ analogously to the case of  a symplectic module. By Corollary \ref{Corol:existencia 2-cociclo quadratic} there is $b\in \Bil(V)$ such that $q(x)=b(x,x)$ for all $x\in V$. Let $A=\mathbb{C}_{b}[V]$ and fix for every $g\in O(V,q)$ a map $\eta_g\in C_0^1(V,\mu_n(k))$ such that $b^{g}/b=\delta(\eta_g)$. Then the map 
\begin{align*}
\alpha_g: A&\to A^{(g)}\\
u_x&\mapsto \eta_gu_{g(x)},
\end{align*}
is an isomorphism of $V$-algebra. The crossed product group associated by Proposition \ref{weil action} which we denote as $$\Ps(V,q):=\widehat{V}\#_\theta O(V,q),$$ and call the pseudo-symplectic group, see \cite{pseudosymplectic} and \cite{pseudodos}.

The Weil representation of  $\Ps(V,q)$ is the projective representation associated to the Weil action of $\Ps(V,q)$ on the simple algebra $A=\mathbb{C}_{b}[V]$.

\begin{obs}
\begin{itemize}
\item If $(V,q)$ is a quadratic module with $V$  homogeneous of exponent $n$, then the Weil representation of $\Ps(V,q)$ is defined over $\mathbb{Z}[e^{\frac{2\pi i}{n}}]$.
\item  The pseudo-symplectic group $\Ps(V,q)$ is a proper subgroup of the Affine pseudo-symplectic  $\APs(V,\omega_q)$ and the Weil representation of  $ \APs(V,\omega_q)$ is an extension of the Weil representation of $\Ps(V,q)$.
\end{itemize}
\end{obs}

\section{Examples of non-isomorphic isocategorical groups and  Weil representations}\label{Section:Examples of non-isomorphic isocategorical groups and  Weil representations}

%Other concrete example of a pair of non-isomorphic isocategorical groups over $\mathbb{C}$ of order $2^6$ was constructed in \cite{Izumi-Kosali}. These groups of order $2^6$ are non-isocategorical over $\mathbb{R}$, since the Galois algebra associated have the form $A=\Ind_{H}^G(\mathbb{C}_\alpha H)$, where $H\cong\mathbb{Z}_4\oplus\mathbb{Z}_4$, then by Proposition \ref{teorema primera obstr}, the Galois algebra $A$ does not admit a real form.

The goal of this section is to construct some concrete examples of
non-isomorphic isocategorical groups.

\begin{lem}
Let $S$ be a finite abelian group and $(S,K,N,\sigma,\gamma)$ a torsor datum over $S$ with associated simple $S$-Galois algebra $B:=A(K_\sigma[N],\gamma)$. Then 
$$\Aut_S(B)\cong \widehat{N}\oplus \Gal(K|k)$$and 
$$\operatorname{St}(B)\cong \operatorname{St}([\sigma,\gamma]):=\{g\in \Aut_N(S): [(\sigma,\gamma)]=[(\sigma^g,\gamma^g)]\in H^2_{\Gal(K|k)}(N,K^*) \},$$where $\Aut_N(S)=\{g\in \Aut(S)\mid g|_N\in \Aut(N)\}$.
\end{lem}
\begin{proof}
The isomorphism $\Aut_S(B)\cong \widehat{N}\oplus \Gal(K|k)$ follows  from Lemma \ref{Lema:descrip automorfismo simple Galois }. The isomorphism $\operatorname{St}(B)\cong \operatorname{St}([\sigma,\gamma])$ follows from the proof of Theorem \ref{main result 2}.
\end{proof}
\begin{teor}\label{Teor:Main4}
Let $S$ be a finite abelian group and $(S,K,N,\sigma,\gamma)$ a torsor datum over $S$, with associated simple $S$-Galois algebra $B:=A(K_\sigma[N],\gamma)$. Then the semidirect product group $$(\widehat{N}\oplus \Gal(K|k))\rtimes  \operatorname{St}([\sigma,\gamma]) $$ and the crossed product group  $$(\widehat{N}\oplus \Gal(K|k))\#_\theta \operatorname{St}([\sigma,\gamma])$$ (See Proposition-Definition \ref{weil action}) are isocategorical over $k$.
\end{teor}
\begin{proof}
Follows from Theorem \ref{main result 2} and Proposition \ref{Prop:Isomorfismo Automorfismo Galois y auto grupoide equivariant}.
\end{proof}

Let $(V,q)$ be a linear quadratic space over a finite field of characteristic $p$. Let $(V,\widehat{q})$ be the the quadratic module associated by Example \ref{Ejemplo: modulo cuadratico de spacop cuadratico}. The \textit{linear} pseudo-symplectic group is defined as $$\Ps_k(V,q):=\widehat{V}\#_{\theta}O_k(V,q)\subset \widehat{V}\#_{\theta}O(V,\widehat{q})= \Ps(V,\widehat{q})$$ and the Weil representation of $\Ps_k(V,q)$ is by definition the restriction of the Weil representation of $\Ps(V,\widehat{q})$.

If $(V, \langle -,-\rangle)$ is a symplectic space over a finite field $k$, by Example \ref{Ejem:smplectico lineal a simplecto modulo} there is an associated symplectic module $(V,\omega)$  and  $\Sp_k(V)\subset  \Sp(V,\omega)$,  where $\Sp_k(V)$ is the the \textit{linear} symplectic group.

The \textit{linear} affine pseudo-symplectic group is defined as $$\APs_k(V):=\widehat{V}\#_{\theta}\Sp_k(V)\subseteq \APs(V,\omega)$$ and its Weil representation is the restriction of the Weil representation of $\APs(V,\omega)$.

 We will denote by $\Omega_k(V,q)$ the subgroup of index 2 in $O_k(V,q)$, for which Dickson invariant is zero, see \cite{Grove} for details and the basic properties of $\Omega_k(V,q)$.
Next proposition is a generalization of \cite[Theorem 1]{Grupo-Extra-special}  to arbitrary finite fields of characteristic two.  

\begin{prop}\label{Proposition:generalizacion Teorema Extra especial}
Let $k$ be a finite field of characteristic two and $(V,q)$ a quadratic space over $k$, where $\dim_k(V)=2n$ and $n\geq 4$. Then the exact sequences 
$$0\to \widehat{V}\to \APs_k(V)\to \Sp_k(V,\omega_q)\to 1,$$ $$0\to \widehat{V}\to \Ps_k(V,q)\to O_k(V,q)\to 1,$$ 
$$0\to \widehat{V}\to \widehat{V}\#_{\theta}\Omega_k(V,q)\to \Omega_k(V,q)\to 1,$$
are nonsplit.
\end{prop}
\begin{proof}
The proof consists in apply \cite[Theorem 0]{Grupo-Extra-special} to our exact sequences. We need to show  that there exists a subgroup $W=\{x,y,z,0\}\subset \widehat{V}$ such that  $x$ and $y$ are singular (that is, $q(x)=q(y)=0$), $z=x+y$ is non-singular (that is, $q(z)\neq 0$) and there exists a subgroup $G\subset \Omega_k(V,q)\subset O_k(V,q)\subset \Sp_k(V,\omega_q)$ satisfying 
\begin{enumerate}
\item $G$ fixes $z$,
\item $G$ has an involution $t$, with $t(x)=y$,
\item $G$ has no subgroup of index 2.
\end{enumerate}  

Let $(x,y)\subset V$ an hyperbolyc pair, that is, $q(x)=q(y)=0$ and $B_q(x,y)=1$. Thus, $z=x+y$ is nonsingular. Let $G=\operatorname{Stab}_z(\Omega_k(V,q))$ and $V=\langle z \rangle_k\oplus H$. Then,  the canonical map $G\to O_k(H,q)$ is an isomorphism and  since $\dim_k(H)$ is odd, $O_k(H,q)\cong \Sp(2n-2)$ (see \cite[Proposition 4.1.7]{APA}). Therefore, we can find an involution $t\in G$ such that  $t(x)=y$. Since $n\geq 4$, $\Sp(2n-2)$ is simple. Thus,  $G$ has no subgroups of index 2.
\end{proof}

\begin{prop}\label{Prop:No isocategorical sobre R}
Let $(V,q)$ be a quadratic spaces over a finite field of characteristic two with $\dim_{k}(V)=2n$ and $n\geq 4$. The affine symplectic group $\ASp_k(V)$ and the affine pseudo-symplectic group $\APs_k(V,\omega_q)$ are isocategorical over $\mathbb{C}$ but not over  $\mathbb{R}$.
\end{prop}
\begin{proof}
If we view the skew-symmetric form $\omega_q$ as an element in $Z^2(V,\mathbb{C}^*)$, then $\operatorname{St}([\sigma])=\Sp(V,\omega_q)$. Therefore, by Theorem \ref{Teor:Main4} the groups $\ASp_k(V)$ and  $\APs_k(V,\omega_q)$ are isocategorical over the field of complex numbers.

Let $S\lhd \ASp(V)$ be an abelian normal subgroup. Since $\Sp(V)$ is a non-abelian simple group and $V\lhd SV\lhd \ASp(V)$, it follows that $S/S\cap V \cong SV/V\lhd\Sp(V)$, thus $S\subset V$. Since $S\lhd \ASp(V)$, it follows that $S=V$. Thus the unique normal abelian  subgroup of $\ASp(V)$ is $V$. Hence, do not exist a semi-real torsor datum over $\ASp(V)$. Let $(V,\mu)$, where $\mu \in H^2(V,\mu_2)$ is a non-degenerate cohomology class. By Lemma \ref{Lemma1:cohomologia ortogonales} and Lemma \ref{Lemma2:cohomologia ortogonales}, $\operatorname{St}([\mu])=O_k(V,q)$ and the orthogonal group is a proper subgroup of the symplectic group. Therefore, the pair $(V,\mu)$ is not a real torsor. If follows that   $\ASp(V)$ does not have real or semireal torsor datum, hence by Theorem \ref{main result 3} every group isocategorical to $\ASp(V)$ group  is isomorphic to  $\ASp(V)$.
\end{proof}
\begin{prop}\label{Prop:Pseudo simplect y affine orthogonal son isocat}
Let $(V,q)$ be a quadratic linear space over a finite field of characteristic two. The following pairs of groups are concrete examples of non-isomorphic groups isocategorical over $\mathbb{Q}$.
\begin{enumerate}
\item $\widehat{V}\rtimes O_k(V,q)$ and $\Ps(V,q)$, where $\dim_{k}(V)=2n$, $n\geq 4$,
\item $\widehat{V}\rtimes \Omega_k(V,q)$ and $\widehat{V}\#_\theta\Omega_k(V,q)$, where $\dim_{k}(V)=2n$, $n\geq 4$,
\item $\mathbb{F}_2^6\rtimes P_2$ and  $\mathbb{F}_2^6\#_\theta P_2$, where $P_2 $ is a  Sylow 2-subgroup of $O(\mathbb{F}_2^6,q)$.
\end{enumerate}
\end{prop}
\begin{proof}
By Theorem \ref{Teor:Main4} we only need to check that  the pairs of group are non-isomorphic. By Proposition \ref{Proposition:generalizacion Teorema Extra especial}, the groups in $(2)$ and $(3)$ are non-isomorphic.

Now, we need to see that the groups  $\mathbb{F}_2^6\rtimes P_2$ and  $\mathbb{F}_2^6\#_\theta P_2$ are non-isomorphic. This can be done easily using the function \textsf{IsIsomorphicPGroup(G,R)}, included in the \gap~package \anupq, \cite{GAP4}.

\end{proof}
\begin{obs}
For the construction of $\mathbb{F}_2^6\rtimes P_2$ and  $\mathbb{F}_2^6\#_\theta P_2$ in \gap\  is useful to know that the exact sequence $$1\to \widehat{V}\to \Ps_{\mathbb{F}_2}(V,q)\to O_{\mathbb{F}_2}(V,q_c)\to 1,$$ is isomorphic to the exact sequence $$1\to \text{Inn}(E)\to \Aut(E)\to \text{Out}(E)\to 1,$$where $E$ is an extra-special 2-group  and  $ \text{Inn}(E), \text{Out}(E)$  are the groups of inner and outer automorphisms of $E$ respectively, see \cite{Grupo-Extra-special}.
\end{obs}

\subsection{Semi-real torsors data and real Weil representations associated to finite symplectic modules}
The aim of this section is to describe a  systematic way  of constructing  torsor data with $\Gal(K|k)$ non trivial.

Let $N$ be a finite  abelian group, $k$ be a field with a primitive root of unity of order the exponent of $N$, $\sigma \in Z^2(N,k^*)$  a non degenerate 2-cocycle and let $k\subset K$ be an abelian Galois field extension.

For the abelian group $S:=N \oplus \Gal(K|k)$ we define the torsor datum $(S,K,N,\gamma_\sigma)$, where $\gamma_\sigma:S\times N\to k^*$ is the pairing defined by $\gamma|_{\Gal(K|k)\times N}=1$ and $\gamma_\sigma|_{N\times N}=\Alt(\sigma)$.

Let us define $$X=\{g\in \Aut(N)| [\sigma]=[\sigma^g] \text{ as elements in }  H^2(N,K^*) \}$$ 
and $$Y=\{g\in \Aut_N(S)| [(\sigma,\gamma_\sigma)]^g=[(\sigma,\gamma_\sigma)] \text{ as elements in }  H^2_S(N,K^*)\},$$ where $\Aut_N(S)=\{g\in \Aut(S) | g|_N\in \Aut(N)\}$.
\begin{teor}\label{Teor:Ultimo}
The restriction map $$r:Y\to X, g\mapsto g|_N$$ defines an exact sequence of groups $$1\to \Aut(\Gal(K|k))\to Y\to X\to 1.$$
\end{teor}
\begin{proof}
Let us first prove that the kernel of  $r:Y\to X, g\mapsto g|_N$ is isomorphic to $ \Aut(\Gal(K|k))$. Let $g\in Y$ and $\eta_g:N\to K^*$ such that $(\sigma^g,\gamma_\sigma^g)=\partial(\eta_g)(\sigma,\gamma_\sigma)$. If $g|_N=\id_N$, thus $\eta_g:N\to K^*$ is a character and $\eta_g(x)\in k^*$ for all $x\in N$. Therefore,  $ \gamma(g(a),x)=\gamma(a,x)$ for all $a \in \Gal(K|k), x\in N$. Thus, $\gamma(g(a)a^{-1},x)=1$ for all $a \in \Gal(K|k), x\in N$. Therefore, $g(a)a^{-1}\in \Gal(K|k)$, that is, $g(a)\in \Gal(K|k)$ for all $a\in \Gal(K|k)$. 

Now we want to show that $r$ is surjective. Let $g\in Y$ and $\eta_g:N\to K^*$ such that $\sigma^g/\sigma=\delta(\eta_g).$ For each $a \in \Gal(K|k)$, the map $\eta_g/a(\eta_g)$ is a character that does not depend on  the choice of $\eta_g$.   Then, there exists a unique $n(g,a)\in N$ such that $\eta_g/a(\eta_g)=\Alt(\alpha)(n(g,a),g(x))$ for all $x\in N$.

Since 
 \begin{align*}
 \Alt(\sigma)(n(g,a)n(g,b),g(x)) &= (\eta_g/a(\eta_g))\eta_g/b(\eta_g)\\
 &=(\eta_g/a(\eta_g))a\eta_g/ab(\eta_g)\\
 &=\eta_g/ab\eta_g\\
 &=  \Alt(\sigma)(n(g,ab),g(x))
 \end{align*}it follows that $n(g,ab)=n(g,a)n(g,b)$, where we have used that $\eta/b\eta(x)\in k^*$ for all $x\in N, b\in A$. Define $g'\in \Aut_N(N\oplus A)$ by $g'(x\oplus a)= g(x)n(g,a)\oplus a$. Then 
 
 \begin{align*}
\frac{\gamma^{g'}}{\gamma}(x\oplus a,y) &= \frac{\gamma(g'(x\oplus a),g(y))}{\gamma(x\oplus a,y)}\\ 
&= \frac{\gamma(g(x)n(g,a)\oplus a,g(y)))}{\gamma(x,y)\gamma(a,y)}\\
&= \frac{\gamma(g(x),g(y)) \gamma (n(g,a),g(y))\gamma(a,g(y))}{\gamma(x,y)\gamma(a,y)}\\
&=\gamma (n(g,a),g(y))=\frac{\eta_g}{a\eta_g}(y),
 \end{align*}
so $g'\in Y$ and $r(g')=g$.
\end{proof}

\begin{obs}
\begin{itemize}
\item The group $Y$ only depends on $[\sigma] \in H^2(N,K^*)$.
\item If $\Gal(K|k)=\mathbb{Z}/2Z$, then $X\cong Y$.
\end{itemize}

\end{obs}

Let $(V,\omega)$ be a symplectic module of exponent two, for example, the  symplectic module associated to a symplectic linear space over finite field of characteristic two, see Example \ref{Ejemplo: modulo cuadratico de spacop cuadratico}. Since $\omega \in Z^2(V,\mathbb{R}^*)\subset Z^2(V,\mathbb{C}^*)$, it follows from by Theorem \ref{Teor:Ultimo} that the tuple $$(\widehat{V}\oplus \Gal(\mathbb{C}|\mathbb{R}), \mathbb{C},\omega,\gamma_\omega)$$ is a semi-real torsor datum over the group $$(\widehat{V}\oplus \Gal(\mathbb{C}|\mathbb{R})) \rtimes \Sp(V,\omega).$$ The  crossed product group $$(\widehat{V}\oplus \Gal(\mathbb{C}|\mathbb{R})) \#_\theta \Sp(V,\omega)$$ (See Proposition-Definition \ref{weil action}) contains the pseudo-symplectic groups and it is isocategorical over $\mathbb{Q}$ to $(\widehat{V}\oplus \Gal(\mathbb{C}|\mathbb{R})) \rtimes \Sp(V,\omega).$

The group $(\widehat{V}\oplus \Gal(\mathbb{C}|\mathbb{R})) \#_\theta \Sp(V,\omega)$ acts by algebra automorphisms on the rational simple algebra $\mathbb{Q}_{\omega}[V]$.   Once a simple module of $\mathbb{Q}_{\omega}[V]$ is fixed, we have a \textit{rational} projective representation of  $(\widehat{V}\oplus \Gal(\mathbb{C}|\mathbb{R})) \#_\theta\Sp(V,\omega)$  that extends the  rational Weil representations of the pseudo-symplectic groups $\Ps(V,q)$.

\bigbreak
\textbf{Acknowledgment.} The author would like to thank the Mathematics Departments at MIT where part of this work was carried out. The author is grateful to Paul Bressler and Pavel Etingof for  useful discussions. This research was partially supported by the FAPA funds from Vicerrector\'{i}a de Investigaciones de la Universidad
de los Andes.

%\bibliographystyle{alpha}

%\bibliography{Multiplicative}
\def\cprime{$'$}

\end{document}